\documentclass[a4paper,11pt,leqno]{article}
\makeatletter
\usepackage{amsfonts,delarray,amssymb,amsmath,amsthm,a4,a4wide,color}
\usepackage{amsmath, euscript,xcolor,color,epsfig,graphics,dsfont}
\usepackage{graphics}

\usepackage[ansinew]{inputenc}

\newtheorem{theo}{Theorem}
\newtheorem{pro}{Proposition}[section]
\newtheorem{lem}[pro]{Lemma}
\newtheorem{coro}[pro]{Corollary}
\newtheorem{remark}[pro]{Remark}
\newtheorem{defi}[pro]{Definition}

\include{defs}

\newcommand{\ed}[1]{{\color{black}{#1}}}
\newcommand{\edd}[1]{{\color{black}{#1}}}

\def\Xint#1{\mathchoice
   {\XXint\displaystyle\textstyle{#1}}%
   {\XXint\textstyle\scriptstyle{#1}}%
   {\XXint\scriptstyle\scriptscriptstyle{#1}}%
   {\XXint\scriptscriptstyle\scriptscriptstyle{#1}}%
   \!\int}
\def\XXint#1#2#3{{\setbox0=\hbox{$#1{#2#3}{\int}$}
     \vcenter{\hbox{$#2#3$}}\kern-.5\wd0}}
\def\dashint{\Xint-}

\DeclareMathOperator{\supp}{Supp}

\def\({\left(}
\def\){\right)}
\def\1{\mathbf{1}}
\def\admissible{{\mathcal{A}_m}}
\def\ainfty{{\mathcal{A}_1}}

\def\bm{{m_0}}

\def\bm{{m_0}}

\def\D{\displaystyle}

\def\B{{\mathcal{B}}}

\def\curl{{\rm curl\,}}

\def\diam{\mathrm{diam}\ }
\def\dist{\text{dist}\ }
\def\div{\mathrm{div} \ }

\def\dt0{{{\frac{d}{dt}}_{|t=0}}}

\def\D{\displaystyle}

\def\E{{\Sigma}}

\def\ep{\varepsilon}

\def\f{{\mathbf f}}
\def\F{{F_n}}

\def\gF{{\mathbf F}}

\def\bnun{{\nu_n}}
\def\bgn{{g_n}}
\def\bjn{{\j_n}}

\def\hal{\frac{1}{2}}

\def\I{{\mathcal{F}}}

\def\indic{\mathbf{1}}
\def\indic{\mathbf{1}}

\def\K{K_n^\beta}

\def\lep{{|\mathrm{log }\ \ep|}}

\def\loc{{\text{\rm loc}}}
\def\Lp{{L^p_\loc(\mr^2,\mr^2)}}

\def\l|{\left|}

\def\mc{\mathbb{C}}

\def\mn{\mathbb{N}}
\def\mn{\mathbb{N}}
\def\mr{\mathbb{R}}
\def\mo{{\mu_0}}

\def\mz{\mathbb{Z}}
\def\nab{\nabla}

\def\np{\nab^{\perp}}

\def\P{\mathcal{P}}
\def\p{\partial}
\def\Q{{\mathbb{P}_n^\beta}}
\def\radon{\mathcal{M}}

\def\ro{\rho}
\def\r|{\right|}

\def\sm{\setminus}

\def\supp{\text{Supp}}

\def\T{{\mathbb{T}}}

\def\vp{\varphi}

\def\w{{w_n}}

\def\x{{\xi_n}}

\def\z{\zeta}
\def\Z{Z_n^\beta}

\def\dr{{\delta_\mr}}
\def\tP{\tilde{P}}

\def\Lj{{\Gamma(\j)}}

\def\wb{{\overline{m}}}
\def\bw{{\underline{m}}}
\def\xbf{{\mathbf x}}

\def\j{{E}}

\def\wQ{{\tilde{\mathbb{P}}_n^\beta}}
\def\tW{\widetilde{W}}

\numberwithin{equation}{section}

\title{1D Log Gases and the Renormalized Energy : Crystallization at Vanishing Temperature}

\author{Etienne Sandier and Sylvia Serfaty}
\begin{document}
\maketitle

\begin{abstract}
We study the statistical mechanics of a one-dimensional log gas or $\beta$-ensemble with general potential and arbitrary $\beta$, the inverse of temperature, according to the method we introduced for two-dimensional Coulomb gases in \cite{ma2d}.
Such ensembles correspond to random matrix models in some particular cases.
 The formal limit $\beta=\infty$ corresponds to ``weighted Fekete sets" and is also treated.

 We introduce a one-dimensional version of the ``renormalized energy" of \cite{ss}, measuring the total logarithmic interaction of an infinite set of points on the real line in a uniform neutralizing background. We show that this energy is minimized when the points are on a lattice.

 By a suitable splitting of the Hamiltonian we connect the full statistical mechanics problem to this renormalized energy $W$, and this allows us to obtain new results on the distribution of the points at the microscopic scale: in particular we show that configurations whose $W$ is above a certain threshhold (which tends to $\min W$ as $\beta\to \infty$) have exponentially small probability. This shows that the configurations have increasing order  and crystallize as the temperature goes to zero.
 \end{abstract}

\section{Introduction}
In \cite{ma2d} we studied the  statistical mechanics
 of a 2D classical Coulomb gas (or two-dimensional plasma) via the tool of the ``renormalized energy" $W$ introduced in \cite{ss}, a particular case of which is the Ginibre ensemble in random matrix theory.

In this paper we are interested in doing the analogue in one dimension, i.e. first defining a ``renormalized energy" for points on the real line, and applying this tool to the study of the classical log gases or $\beta$-ensembles, i.e. to probability laws of the form
\begin{equation}
\label{loi}
d\Q(x_1, \dots, x_n)=   \frac{1}{\Z} e^{-\frac{\beta}{2} \w (x_1, \dots, x_n)}dx_1\, \dots dx_n\end{equation} where $\Z$ is the associated  partition function, i.e. a normalizing factor such that $\Q$ is a probability, and
\begin{equation}
\label{wn}
\w (x_1, \dots, x_n)= -  \sum_{i \neq j} \log |x_i-x_j| +n  \sum_{i=1}^n V(x_i).\end{equation}
Here   the $x_i$'s belong to  $\mr$, $\beta>0$ is a parameter corresponding to (the inverse of) a temperature, and $V$ is a relatively arbitrary potential, satisfying some growth conditions. For a general presentation, we refer to the textbook \cite{forrester}. Minimizers of $\w$ are also called ``weighted Fekete sets" and arise in interpolation, cf. \cite{safftotik}.

There is an abundant literature on the random matrix aspects of this problem (the connection  was first pointed out by Wigner and  Dyson \cite{wigner,dyson}), which is the main motivation for studying log gases. Indeed, for the quadratic potential $V(x)=x^2/2$, particular cases of $\beta$ correspond to the most famous random matrix ensembles: for
 $\beta=1$  the law $\Q$ is the law of eigenvalues of matrices of  the Gaussian Orthogonal Ensemble (GOE), while  for $\beta=2$ it corresponds to the Gaussian Unitary Ensemble (GUE), for general reference see \cite{forrester,agz,mehta}.  For  $V(x)$ still quadratic, general $\beta$'s have been shown to correspond to tri-diagonal random matrix ensembles, cf. \cite{de,abf}.
 This observation allowed Valk\'o and Vir\'ag \cite{vv}  to derive the sine-$\beta$ processes as the local spacing distributions of these ensembles.
  When $\beta=2$ and $V(x)$ is more general, the model corresponds  (up to minor modification) to other determinantal processes called  orthogonal polynomial ensembles (see e.g. \cite{koe} for a review).

 The study of $\Q$ via the random matrix aspect is generally based on explicit formulas for correlation functions and local statistics, obtained via orthogonal polynomials, as pioneered by Gaudin, Mehta, Dyson, cf. \cite{mehta,deift,dg}.
\ed{We are interested here in the more general setting of general $\beta$ and $V$, with equilibrium measures for the empirical distribution of the eigenvalues whose support can have several connected components, also called the ``multi-cut regime" as opposed to the ``one-cut regime."  One class of recent results in this direction are those of Borot-Guionnet  and Shcherbina  who prove in particular partition functions expansions in  the case of the one-cut regime with general $V$ \cite{bg1,sh1}  or the case of the multi-cut regime with analytic $V$ \cite{bg2,sh2} (see references therein for prior results). Another is those by Bourgade-Erd\"os-Yau  \cite{bey1, bey2} who prove universality (i.e. independence with respect to $V$) of the eigenvalue gap distribution for analytic $V$ (see also the recent result of Bekerman-Figalli-Guionnet \cite{bfg} obtained by a transport method in the one-cut regime with $V\in C^{31}$). }
\smallskip 

The results and the method  here are counterparts of those  obtained in \cite{ma2d} for $x_1, \dots, x_n$ belonging to $\mr^2$, in other words the two-dimensional Coulomb gas (this corresponds for $V$ quadratic and $\beta=2$ to the Ginibre ensemble of non-hermitian Gaussian random  matrices). The study in \cite{ma2d} relied on relating the Hamiltonian $\w$ to a Coulomb ``renormalized energy" $W$ introduced in \cite{ss} in the context of Ginzburg-Landau vortices.
This relied crucially on the fact that the logarithm is the Coulomb kernel in two dimensions, or in other words the fundamental solution to the Laplacian. When looking at the situation in one dimension, i.e. the present situation of the 1D log-gas, the logarithmic kernel  is no longer the Coulomb kernel, and it is not a priori clear that anything similar to the study in two dimensions can work.  Note that the 1D Coulomb gas, corresponding to $\Q$ where the logarithmic interaction is replaced by the 1D Coulomb kernel $|x|$, has been studied, notably by Lenard \cite{len1,len2}, Brascamp-Lieb \cite{bl}, Aizenman-Martin \cite{am}.  The situation there is rendered again more accessible by the Coulomb nature of the interaction and its less singular character. In particular \cite{bl} prove cristallization (i.e. that the points tend to arrange along a regular lattice) in the limit of a small temperature, we will get a similar result for the log-gas.

The starting point of our study is that even though the logarithmic kernel is not Coulombic in dimension 1, we can view the particles on the real line as embedded into the two-dimensional plane and interacting as Coulomb charges there. This provides a way of defining an analogue of the ``renormalized energy" of \cite{ss}  in the one-dimensional setting, still called $W$, which goes ``via" the two-dimensional plane \ed{and is a way of computing the $L^2$ norm of the Stieltjes transform, cf. Remark~\ref{stieljes} below. }

Once this is accomplished, we connect in the same manner as \cite{ma2d} the Hamiltonian $\w$ to the renormalized energy $W$ via a ``splitting formula" (cf. Lemma \ref{lemsplit} below), and we obtain the counterparts results to \cite{ma2d}, valid with our relatively weak assumptions on $V$:
\begin{itemize}
\item
a next-order expansion of the partition function in terms of $n$ and $\beta$, cf. Theorem \ref{valz}. 
\item the proof that the minimum of $W$ is achieved by the one-dimensional regular lattice $\mz$, called the ``clock distribution" in the context of orthogonal polynomial ensembles \cite{simon}. This is in contrast with the dimension 2 where the identification of minimizers of $W$ is still open (but conjectured to be ``Abrikosov" triangular lattices.)
\item  the proof that ground states of $\w$, or  ``weighted Fekete sets", converge to  minimizers of $W$ and hence to \ed{cristalline states}, cf. Theorem \ref{th2}.
\item A large deviations type result which shows that events with high $W$ become less and less likely as $\beta \to \infty$, proving in particular the  crystallization as the temperature tends to $0$.
\end{itemize}
Our renormalized energy $W$, which serves to prove the crystallization, also  appears (like its two-dimensional version) to be a measurement of ``order" of a configuration at the microscopic scale $1/n$. \ed{This is more precisely quantified in \cite{leble}.} What we show here is that there is more and more order (or rigidity) in the log gas, as the temperature gets small. Of course, as already mentioned, it is known that eigenvalues of random matrices, even of general Wigner matrices, should be regularly spaced, and   \cite{vv,bey1,bey2} showed that this could be extended to general $V$'s.
 Our results approach this question sort of orthogonally, by exhibiting a unique number which  measures the average rigidity.
(Note that in \cite{bs} the second author and Borodin used $W$ as a way of quantifying the order of random point processes, in particular those arising as local limits in random matrix theory.)

\ed{Cristallization was already known in some particular or related settings. One is  the case where $V$ is quadratic,  for which the $\beta\to \infty$ limits of the  eigenvalues -- in other words the weighted  Fekete points --  are also zeroes of Hermite  polynomials, which are known to have  the clock distribution (see e.g.  \cite{als}). The second is the case of the  $\beta$-Jacobi ensemble \cite{vv}. }

Our study here differs technically  from the two-dimensional one in two ways: the first one is in the definition of $W$ by embedding the problem into the plane, as already mentioned. The second one is more subtle: in both settings a crucial ingredient in the analysis is to reduce the evaluation of the interactions to an extensive quantity (instead of sums of pairwise Coulomb interactions); that quantity is essentially the $L^2$ norm of the electric field generated by the Coulomb charges, or equivalently of the Stieljes transform of the point distribution. Test-configurations can be built and their energy evaluated by ``copying and pasting", provided a cut-off procedure is devised: it consists essentially in taking a given electric  field and making it vanish on the boundary of a given  box while not changing its energy too much. In physical terms, this corresponds to {\it screening} the field. The point is that screening is much easier in two dimensions than in one  dimension, because in two dimensions there is more geometric flexibility to
  move charges around.   We found that in fact, in dimension 1, not all configurations with finite energy can be effectively screened. However, we also found that generic ``good" configurations can be, and this suffices for our purposes. The screening construction, which is different from the two-dimensional one, is one of the main difficulties here, and forms a large part of the \smallskip paper.

The rest of the introduction is organized as follows: \edd{We begin by introducing the equilibrium measure (i.e. the minimizer of the mean-field limiting Hamiltonian) and known facts concerning it, in the next two sections we describe the central objects in our analysis, i.e. the marked electric field process and the renormalized energy $W$. Then we state the results which connect $\w$ to $W$: the ``splitting formula", and the Gamma-convergence lower and upper bounds. Finally, in Section \ref{sec1.3} we state our main results about Fekete points and the 1D Coulomb gas.}

\subsection{The spectral and equilibrium measures and our assumptions}\label{sec1.2}
The Hamiltonian \eqref{wn} is written in the mean-field scaling. The limiting ``mean-field" limiting energy (also called Voiculescu's noncommutative entropy in the context of  random matrices, cf. e.g. \cite{agz} and references therein) is 
 \begin{equation}\label{Ib}
\I (\mu)= \int_{\mr\times \mr} - \log |x-y|\, d\mu(x) \, d\mu(y) + \int_{\mr} V(x)\, d\mu(x), \end{equation}
it is well known  (cf. \cite{safftotik}) that it has a unique minimizer, called the (Frostman) equilibrium measure, which we will denote $\mo$. 
 It is not hard to prove that the ``spectral measure" (so-called in the context of random matrices) $\nu_n = \frac{1}{n} \sum_{i=1}^n \delta_{x_i}$ converges to  $\mo$. The sense of convergence usually proven is
  \begin{equation}\mathbb{P}\( \forall f \in C_b( \mc, \mr) , \int f\, d\nu_n \to\int   f\, d\mu_0 \) =1\end{equation}
For example, for the case of the GUE i.e. when $V(x)=|x|^2$ and $\beta=1$, the correspond distribution $\mo$ is simply Wigner's ``semi-circle law" $\ro(x)=\frac{1}{2\pi}\sqrt{4-x^2}\indic_{|x|<2}$, cf. \cite{wigner,mehta}.
 A stronger result was proven in \cite{bg} for all $\beta$ (cf. \cite{agz} for the case of general $V$): it estimates the large deviations from this convergence and shows that $\I$ is the appropriate rate function. The result can  be written:
 \begin{theo}[Ben Arous - Guionnet \cite{bg}]\label{thbg}
 Let $\beta>0$, and denote by $\wQ$ the image of the law \eqref{loi} by the map $(x_1,\dots,x_n)\mapsto \nu_n$, where $\nu_n=\frac{1}{n}\sum_{i=1}^n \delta_{x_i}$. Then for any subset  $A$ of the set of  probability measures on $\mr$ (endowed with the topology of weak convergence), we have
$$
-\inf_{\mu \in \mathring{A}} \widetilde \I(\mu) \leq \liminf_{n\to \infty} \frac{1}{n^2} \log \wQ(A) \leq \limsup_{n\to\infty} \frac{1}{n^2} \log \wQ(A) \leq -\inf_{\mu\in \bar A} \widetilde \I(\mu),
$$
where $\widetilde \I =\frac{\beta}{2}( \I - \min \I)$.
\end{theo}
The Central Limit Theorem for (macroscopic) fluctuations from the law $\mo$ was proved by Johansson \cite{joha}.
\smallskip 

Let us now state a few facts that we will need about the equilibrium measure  $\mo$, for which we refer to \cite{safftotik}: $\mo$  is   characterized by the fact that
  there exists a constant $c$ (depending on $V$) such that
\begin{equation}\label{optcondmo}
U^\mo + \frac{V}{2}= c\  \text{quasi-everywhere in the support of $\mo$, and} \  U^\mo + \frac{V}{2}\ge c \text{ quasi-everywhere}\end{equation}
 where for any $\mu$, $U^\mu$ is the potential generated by $\mu$, defined by
\begin{equation}
U^\mu(x)= - \int_{\mr} \log |x-y|\, d\mu(y) .\end{equation}
 We also define \begin{equation}\label{defz}
 \z= U^\mo+\frac{V}{2}-c \end{equation} where $c$ is the constant in \eqref{optcondmo}. From the above we know that $\z\ge 0$ in $\mr$ and $\z= 0$ in $\E:=\supp(\mo)$. We will make the assumption that $\mu_0$ has a density $m_0$ with respect to the Lebesgue measure, as well as the following additional assumptions:

\begin{equation}\label{assumpV1}     
 \text{$V$ is lower semicontinuous and $\lim_{|x|\to +\infty} \frac{V(x)}{2} - \log|x| = +\infty.$}
\end{equation}
\begin{equation}\label{assumpV2} \E \ \text{is a finite union of closed intervals } \E_1, \dots, \E_M \ \textrm{(multi-cut)}.    \end{equation}
 \begin{equation}\label{assumpV3}
 \text{There exist $\gamma,\wb>0$ such that $\gamma \sqrt{\dist(x,\mr\sm \E)}  \le m_0(x) \le \wb$ for all $x\in\mr$.}\end{equation}
 \begin{equation} \label{assumpV4} m_0\in C^{0,\hal}(\mr).\end{equation}
 \begin{equation}\label{assumpVsupp}\text{There exists} \
 \beta_1>0 \ \text{such that} \ \int_{\mr\backslash [-1,1]} e^{- \beta_1(V/2(x) -\log|x|)} \, dx <+\infty. \end{equation}
  The assumption \eqref{assumpV1} ensures (see \cite{safftotik}) that \eqref{Ib} has
   a minimizer, and that its support $\E$ is compact. Assumptions
   \eqref{assumpV2}--\eqref{assumpV4}  are needed for the construction
   in Section~\ref{sec4}. They could certainly be relaxed  but are  meant to include
   at least the model case of $\mo=\ro$, Wigner's semi-circle law.
   Assumption \eqref{assumpVsupp} is  a supplementary assumption on the growth of $V$ at infinity,
   needed for the case with temperature. It only requires a very mild growth of
   $V/2- \log |x|$, i.e. slightly more than \eqref{assumpV1}.


\edd{\subsection{The marked electric field process}
Theorem~\ref{thbg} describes the asymptotics of $\Q$ as $n\to +\infty$ in terms of the spectral measure $\nu_n = \frac{1}{n} \sum_{i=1}^n \delta_{x_i}$. Our results will  rather use an object which  retains information at the microscopic scale : the marked electric field process. 

More precisely, given any configuration $\xbf= (x_1, \dots, x_n)$, we let $\nu_n= \sum_{i=1}^n \delta_{x_i}$  and  $\nu_n' = \sum_{i=1}^n\delta_{x_i'}$  where the primes denote blown-up quantities ($x'=nx$).  
We set $\bm'(nx) = \bm(x)$, and we denote by $\dr$ denotes the measure of length on  $\mr$ seen as embedded in $\mr^2$, that is
$$\int_{\mr^2} \varphi  \dr  = \int_\mr \varphi(x, 0)\, dx$$
for any smooth compactly supported test function $\vp $ in $\mr^2$.
The  configuration $\xbf$ generates (at the blown-up scale) an electric field  via
\begin{equation}\label{Hp} \j_{\nu_n} :=  - \nab H_n',\quad\text{where}\quad \Delta H_n' = - 2\pi  \(\nu_n'-  \bm'\dr \).\end{equation} 
where  $H_n'$ is understood to be the only solution  of the equation which decays at infinity, which is obtained by convoling the right-hand side with $-\log |x|$. We will sometimes write it as $
 H_n' = - 2\pi \Delta^{-1} \(\nu_n' - \bm' \dr \) = -\log *\(\nu_n' - \bm' \dr\)  $. Here 
We note that from \eqref{Hp}, $\j_{\nu_n}$ satisfies the relation 
\begin{equation}\label{eqsjnu}
\div E_{\nu_n}= 2\pi \(\nu_n'-\bm'\dr\) \quad \text{in} \ \mr^2,\end{equation}
supplemented with the fact that $E_{\nu_n}$ is a gradient.
\ed{

\begin{remark} \label{stieljes}
When considering the Stieljes transform of a (say compactly supported) measure $\mu$ on $\mr$, 
$$S(z)=\int \frac{d\mu(x)}{z-x}, \quad z \in \mathbb{C}$$
one observes that 
$$|S(z)|= |\nabla \log * \mu |.$$
Thus the electric field $E=-\nab \log *\mu$ of the type we introduced is very similar to the Stieltjes transform, in particular they have the same norm. We note however that it seems much easier to take limits in the sense of distributions -- what we will need  to do --   in  \eqref{eqsjnu} than in Stieltjes transforms. 
\end{remark}

}

The field $\j_{\nu_n}$ belongs to $ L^p_\loc(\mr^2,\mr^2) $
 for any  $p\in [1,2)$. Choosing once and for all such a $p$, we define $X:= \Sigma \times L^p_\loc(\mr^2,\mr^2) $ the space of ``marked" electric fields, where the mark $x\in \Sigma$ corresponds to the point where we will center the blow-up.   We denote by 
 $\P(X)$  the space of probability measures on $X$ endowed with the Borel $\sigma$-algebra, where the topology is the usual one on $\mr$ and  the topology of weak convergenceon $L^p_\loc$.
 
We  may now naturally associate  to  each configuration $\xbf= (x_1, \dots, x_n)$ a ``marked electric field distribution" $P_{\nu_n}$ via the map 
\begin{eqnarray}
\label{in}
i_n : &  \mr^n \longrightarrow \P(X) 
\\  \label{pnun} & \xbf \mapsto P_{\nu_n} :=  \dashint_{\E} \delta_{(x,\j_{\nu_n}(nx+\cdot))}\,dx, 
\end{eqnarray}
i.e. $P_{\nu_n} $ is  the push-forward of the normalized Lebesgue measure on $\E$ by $x \mapsto (x, \j_{\nu_n} ({n}x+\cdot)).$
Another way of saying is that each $P_{\nu_n} (x, \cdot)$ is equal to a Dirac at the electric field generated by $\xbf$, after centering at the point $x$. We stress that $P_{\nu_n}$  has nothing  to do with $\Q$, and  is strictly an encoding of a particular configuration $(x_1,\dots,x_n)$. 

The nice feature is that, assuming  a suitable bound on  $\w(x_1, \dots, x_n)$,  the sequence $\{P_{\nu_n}\}_n$ will be proven to be tight as $n\to \infty$, and thus to converge to an element $P$ of  $\P (X)$. From the point of view of  analysis, $P$ may be seen as a family $\{P^x\}_{x\in\E}$ --- the disintegration of $P$ --- each $P^x$ being a probability density describing the possible blow-up limits of the electric field when the blow-up center is near $x$. It is similar to the  Young measure on micropatterns of \cite{albmu}.

When $(x_1, \dots, x_n)$ is random then $P$ also is and, from a probabilistic point of view, $P$  is  an electric field process, or to be more precise an electric field distribution process.

The limiting $P$ will be concentrated on vector fields which are obtained by taking limits in \eqref{eqsjnu} (after centering at $x$), which will be elements  of the following classes:
\begin{defi}\label{defA}Let $m$ be a positive number.
A vector field  $\j$ in $\mr^2$ is  said to 
belong to the admissible class
$\mathcal{A}_m $
 if it is a gradient and 
  \begin{equation}\label{eqj}
\div \j= 2\pi (\nu -m\delta_\mr ) \quad  \text{in} \ \mr^2\end{equation}
where $\nu$ has the form
\begin{equation}\label{eqnu}
\nu=  \sum_{p \in\Lambda} \delta_{p}\quad \text{ for
some discrete set} \  \Lambda\subset\mr \subset \mr^2,\end{equation}  
and
\begin{equation}\label{densbornee}
\frac{    \nu ([-R,R] )  } {R}\quad \text{ is bounded by a constant independent of $R>1$}.
\end{equation}\end{defi}
One should understand the  class $\mathcal{A}_m$ as corresponding to infinite configurations on the real line with density of points $m$. The distribution of points on the real line, seen as positive Dirac charges,  is compensated by a background charge $m\dr$ which is also concentrated on the real line.

The properties satisfied by $P= \lim_{n\to \infty} P_{\nu_n}$ may now be summarized in the following definition:
\begin{defi}[admissible probabilities]\label{admissible}
We say $P\in \P (X)$ is \em{admissible} if 
\begin{itemize}
\item The first marginal of $P$ is the normalized Lebesgue measure on $\E$.
\item It holds for $P$-a.e. $(x, E)$ that $\j \in \mathcal{A}_{m_0(x)}$.
\item $P$ is $T_{\lambda(x)}$-invariant.
\end{itemize} 
\end{defi}
Here $T_{\lambda(x)}$-invariant is a strengthening of translation-invariance, related to the marking:
\begin{defi}[ $T_{\lambda(x)}$-invariance]\label{invariance} We say a probability measure $P$ on $\E \times L^p_\loc(\mr^2,\mr^2)$ is $T_{\lambda(x)}$-invariant if $P$ is invariant  by $(x,\j)\mapsto\(x,\j(\lambda(x)+\cdot)\)$, for any $\lambda(x)$ of class $C^1$  from $\E$ to $\mr$.
\end{defi}

Note that from such an admissible electric field process $P$, and since  $\j \in \mathcal{A}_{m_0(x)}$ implies that $\j$ solves  \eqref{eqj}, one can immediately get a (marked)  point process by taking the push-forward of  $P(x, \j) $ by $\j\mapsto \frac{1}{2\pi} \div \j + m_0(x) \delta_\mr$. This process remembers only the point locations, not the electric field they generate, but we will show (Lemma~\ref{depoint}) that the two are equivalent.
}

\edd{\subsection{The renormalized energy}

In Theorem~\ref{thbg}, large deviations  (at speed $n^{2}$)  from the equilibrium measure $\mu_0$ of the spectral measure $\nu_n$ were described with the rate function based on  the  energy $\I(\mu)$. Our statements concern the next order behavior, and if we try to put them in parallel to Theorem~\ref{thbg},  the electric field distribution replaces the spectral measure as the central object, while  the renormalized energy $W$ that we describe in this section replaces $\I$. 

First we define the renormalized energy of an electric field $\j$. It is adapted from \cite{ss} which considered distribution of charges in the plane, by simply ``embedding" the real line into the plane. 
As above we denote points in $\mr$ by the letter $x$ and points in the plane by $z=(x,y)$.}

\begin{defi}\label{def1}Let $m$ be a nonnegative number. For  any bounded function $\chi$ and any
 $\j$ satisfying a relation of the form \eqref{eqj}--\eqref{eqnu}, 
we let
\begin{equation}\label{WR}W(\j, \chi) = \lim_{\eta\to 0} \(
\hal\int_{\mr^2 \backslash \cup_{p\in\Lambda} B(p,\eta) }\chi
|\j|^2 +  \pi \log \eta \sum_{p\in\Lambda} \chi (p) \)
\end{equation}
and the renormalized energy $W$ is defined by
\begin{equation} \label{Wroi} W(\j)= \limsup_{R \to \infty}
\frac{W(\j, \chi_{R})}{R} ,
\end{equation} where $\{\chi_R\}_{R>0}$ is a family of cut-off functions satisfying 
\begin{equation}
\label{defchi} |\nab \chi_{R}|\le C, \quad \supp(\chi_{R})
\subset [-R/2,R/2] \times \mr, \quad \chi_{R}(z)=1 \ \text{if } |x|<R/2-1 ,\end{equation}
for some $C$ independent of $R$. \end{defi}

\ed{
After this work was completed, a slightly different definition of renormalized energy was proposed in \cite{rs} for points in dimensions $d\ge 2$. A version for dimension one can also be written down, cf. \cite{ps} and this allows to retrieve our results with a  few simplifications in the  proof, more precisely it  suppresses the need for Proposition~\ref{wspread}.}

\begin{remark}
While $W$ in 2D can be viewed as a ``renormalized" way of computing $\|H\|_{H^1(\mr^2)}$, in 1D it amounts rather to a renormalized computation of $\| H\|_{H^{1/2}(\mr)}$ (where $H^s$ denote the fractional Sobolev spaces).  In other words, because the logarithmic kernel is not Coulombic in one-dimension, the associated energy is non-local (and the associated operator is the fractional Laplacian $\Delta^{1/2}$). Augmenting the dimension by 1 allows to make it local and Coulombic again. This well-known harmonic extension idea  seems to be attributed to \cite{so}.\end{remark}

As in the two dimensional case, we have the following properties:
\begin{itemize}
\item[-] The value  of $W$ does not depend on $\{\chi_{R}\}_R$
 as long as  it satisfies \eqref{defchi}.
 \item[-] $W$ is insensitive to compact perturbations of the configuration.

\item[-]
Scaling: it is easy to check that if $\j$ belongs to $\mathcal{A}_m$ then
$\j':=\frac{1}{m} \j(\cdot / m)$ belongs to $\mathcal{A}_1 $ and
\begin{equation}\label{minalai1}
W(\j)= m \(W(\j') - \pi \log m\),\end{equation}
so one may reduce to studying $W$ on $\mathcal{A}_1$.
\item[-] If $\j\in\mathcal A_m$ then in the neighborhood of $p\in\Lambda$ we have $\div \j  = 2\pi(\delta_p - m\dr )$, $\curl \j = 0$, thus we have near $p$ the decomposition $\j(x) = - \nab   \log|x-p| + f(x)$ where $f$ is smooth, and it easily follows that the limit \eqref{WR} exists.  It also follows that $\j$ belongs to $L^p_\loc$ for any $p<2$, as stated above.
\end{itemize}

In the case where \eqref{eqnu} is  satisfied, then there exists at most on $\j$ satisfying \eqref{eqj} and such that $W(\j) <+\infty$. This is the content of the next lemma, and is in contrast with the $2$-dimensional case ---  when the support of $\nu$ is  not constrained to lie on the real line and where the definition of $W$ is modified accordingly --- where \eqref{eqj} and $W(\j)<+\infty$ only determine $\j$ up to constant (see Lemma 3.3 in \cite{ma2d}). \edd{The following lemma is proved in the appendix.}

\begin{lem}\label{depoint}
Let $\j\in \mathcal{A}_m$ be such that $W(\j)<+\infty$. Then any other $\j'$ satisfying \eqref{eqj}--\eqref{eqnu} with the same $\nu$ and $W(\j')<+\infty$, is such that $\j'=\j$. In other words, $W$ only depends on the points.
\end{lem}
By simple considerations similar to \cite[Section 1.2]{ma2d}  this makes $W$ a measurable function of the bounded Radon measure $\nu$.


The following lemma is proven in \cite{bs}, \edd{see also \cite{saffcircle}}, and shows that there is an explicit formula for $W$ in terms of the points when the configuration is assumed to have some periodicity. Here we can reduce to $m=1$ by scaling, as seen above.
\begin{lem}\label{casper}
In the case  $m=1$ and when  the set of points $\Lambda$ is  periodic with respect to some lattice $ N\mz$, then  it can be viewed as a set of $N$ points $a_1, \dots , a_N$ over the torus
$\T_N := \mr/(N\mz)$.  In this case,  by Lemma \ref{depoint}  there exists a unique  $\j$ satisfying \eqref{eqj} and for which $W(\j)<+\infty$. It is periodic and  equal to $\j_{\{a_i\}}= \nab H$, where $H$ is the  solution  on $\T_N$ to   $-\Delta H = 2\pi (\sum_i\delta_{a_i} - \dr)$, and we have  the explicit formula:
\begin{equation} \label{WNlog}
W(\j_{\{a_i\}})=
 -\frac{\pi}{N} \sum_{i \neq j} \log \left|2\sin \frac{\pi(a_i - a_j)}{N} \right|- \pi \log \frac{2\pi}{N}.
 \end{equation}
\end{lem}
As in the two-dimensional case, we can prove that $\min_{\mathcal{A}_m}$ is achieved, but contrarily to the two-dimensional case, the value of the minimum can be explicitly computed: we will prove  the following 
\begin{theo}\label{thmini}
 $\min_{\admissible} W=  - \pi m \log (2\pi  m) $
and this minimum is achieved by the perfect lattice i.e. $\Lambda= \frac{1}{m} \mz$.
\end{theo}
We recall that in dimension 2, it was conjectured in \cite{ss} but not proven, that the minimum value is achieved at the triangular lattice with angles $60^\circ$ (which is shown to achieve the minimum among all lattices), also called the Abrikosov lattice in the context of superconductivity.

The proof of Theorem \ref{thmini} relies on showing that a minimizer can be approximated by configurations which are periodic with period $N \to \infty$ (this result itself relies on the screening construction mentioned at the beginning), and then using a convexity argument to find the  minimizer among periodic configurations with a fixed period via \eqref{WNlog}. 

\edd{The minimizer of $W$ over the class $\mathcal{A}_m$ is not unique, because as already mentioned it suffices to perturb the points of the lattice $m\mz$ in a compact set only, and this leaves $W$ unchanged. However, it is proven by Lebl\'e in \cite{leble} that $W$, once averaged with respect to a translation-invariant probability measures, has a unique minimizer. We now describe more precisely this averaging of $W$ and Lebl\'e's result.

We may extend  $W$ into a function on electric field (or point) processes, as follows:
given any $m>0$, we define
 $$ \overline{W}(P):= \int W(E)\, dP(E)$$ 
over stationary probability measures on $\Lp$ concentrated on the class $\mathcal{A}_m$.
Lebl\'e proves that $\overline{W}$ achieves a unique minimum of value $\min_{\mathcal{A}_m} W$, and 
the unique minimizer is   $P_{\frac1m\mz}$, defined as   the electric field process   associated (via Lemma \ref{casper}) to the point configurations $u+\frac1m\mathbb{Z}$ where $u$ is uniform in $[0,\frac1m]$. In other words, to each $u \in [0,\frac1m]$ we associate the unique (by Lemma \ref{depoint}) periodic electric field $E_{u+ \frac1m \mz}$ such that $\div E= 2\pi( \sum_{p\in \mz} \delta_{u+\frac1m p} - m \delta_\mr)$, and define $P_{\frac1m\mz}$ as the push-forward of the normalized Lebesgue measure on $[0,\frac1m]$ by $u \mapsto E_{u+\frac1m\mz}$.

Lebl\'e's proof is quantitative: he shows the estimate
 \begin{equation}\label{estlebl}
 \left|\int (\ro_2(x,y) - \ro_{2, \mz}(x,y))\varphi (x,y) \right| \le C_\varphi (\overline W(P)+ C)^\hal  (\overline{W}(P)- \min_{\mathcal{A}_m} W)^\hal
 \end{equation} for $\vp \in C^1_c(\mr\times \mr)$, 
where $\ro_{2}$ is the two-point correlation function 
 of the point process associated to $P$ (i.e. given by the push-forward of $P$ by $P\mapsto \frac{1}{2\pi} \div P + m \delta_\mr$) and
 $\ro_{2, \mz}$ is the two-point correlation function associated to the point process $u + \frac1m\mz$ where $u$ follows a uniform law on $[0,\frac1m]$.

We will also need a version of  $W$ for marked electric field processes, in fact it is the one that will play the role of the rate function in our results.
For each  $P\in \P(X)$, we  let
\begin{equation}
\label{Wtilde}
\tW(P) =
 \begin{cases}
\frac{|\E|}{\pi} \int W(E)\, dP(x, E) & \ \text{if $P$ is admissible  } \\
 + \infty & \ \text{otherwise}.
 \end{cases}
 \end{equation}
In view of Theorem \ref{thmini} and the definition of admissible, the minimum of $\tW$ can be guessed to be 
\begin{equation}
\label{defa}
\min \tW=  - \int_\E  m_0(x) \log (2\pi m_0(x))\, dx.\end{equation}
From \cite{leble}, this minimum is uniquely achieved (here the assumption of translation-invariance made in the definition of admissible is the crucial point):
\begin{coro}[\cite{leble}]\label{leble}
The {\it unique} minimizer of $\tW$ on $\mathcal P(X)$ is 
$$P_0=\frac{dx_{|\E}}{|\E|} \otimes  P_{\frac1{m_0(x)} \mathbb{Z}} $$
where $P_{\frac1m\mz}$ has just been defined.
\end{coro}

}

\edd{
\subsection{Link between $\w$ and $W$} 
We are now ready to state the two basic results which link the energies $\w$ and $W$. In the language of Gamma-convergence\footnote{A sequence of functionals $\{f_n\}_n$ Gamma-converges to $f$ if  (i) for any sequence $x_n\to x$, $\liminf_n f_n(x_n)\ge f(x)$  and (ii) for any $x$ there exists a sequence $x_n\to x$ such that $f(x) = \lim_n f_n(x_n)$. See \cite{braides} for an introduction to the subject.} these results establish in essence that the second term in the development of $\w$ by Gamma-convergence is $\tW$ (the first term being $\I$). The consequences for the asymptotics of minimizers of $\w$ and $\Q$ will be stated in the next subsection.

We begin with the following splitting formula which is the starting point to establish this link,} \edd{and which is proved in the appendix}.



\begin{lem}[Splitting formula]\label{lemsplit}
For any $n$, any $x_1, \dots, x_n \in \mr$ the following holds
\begin{equation}\label{idwnbu}
\w(x_1, \dots, x_n)= n^2 \I (\mo)- n \log n + \frac{1}{\pi} W(\nab H_n', \indic_{\mr^2} )  +2 n \sum_{i=1}^n  \z (x_i)\end{equation}
where $H_n'$ is as in \eqref{Hp}, $W$ as in \eqref{WR}, and  $\zeta$ as in \eqref{defz}.
\end{lem}

We may then define
\begin{equation}\label{defF}
\F(\nu)=
 \begin{cases}
 \frac{1}{n}\( \frac{1}{\pi} W(\nab  H_n', \indic_{\mr^2}) + 2n\int_{\mr} \z \, d\nu\) & \ \text{if } \ \nu \ \text{is of the form}  \sum_{i=1}^n \delta_{x_i}\\
 + \infty & \ \text{otherwise}
 \end{cases}
\end{equation}
and also
\begin{equation}\label{fnhat}
\widehat{\F}(\nu) = \F(\nu) - 2\int_{\mr} \z \, d\nu \le \F(\nu)\end{equation}
and we thus have the following rewriting of $\w$:
\begin{equation}\label{lienfw}
\boxed{
\w(x_1, \dots, x_n)=  n^2 \I (\mo)- n \log n + n \F(\nu).}\end{equation}
This allows to separate orders in the limit $n \to \infty$  since one of the main outputs of our analysis is that $\F(\nu)$ is of order $1$.

\edd{ We next state some preliminary results which connect directly $F_n$ and $\widetilde{W}$.  The first result is a lower bound corresponding to the lower-bound part in the definition of Gamma-convergence.
We will  systematically abuse notation by writing $(x_1, \dots, x_n)$ instead of $(x_{1,n}, \dots, x_{n,n})$ and  $\nu_n= \sum_{i=1}^n \delta_{x_i}$ instead of  $\nu_n= \sum_{i=1}^n \delta_{x_{i,n}}$.

\begin{theo}[Lower bound] \label{thlowb}  Let the potential $V$ satisfy assumptions \eqref{assumpV1}, \eqref{assumpV4}.
 Let  $\nu_n= \sum_{i=1}^n \delta_{x_i}$ be a sequence such that $ \widehat{F_n}(\nu_n)  \le C$, and let $P_{\nu_n}$ be associated via 
\eqref{pnun}. 

 Then any subsequence of $\{P_{\nu_n}\}_n$ has a convergent subsequence converging
as $n\to \infty$  to an admissible  probability measure $P\in\P(X)$ and 
\begin{equation}\label{thlow}\liminf_{n \to \infty} \widehat{F_n}(\nu_n)    \ge \tW(P).\end{equation} 
\end{theo}

The second result corresponds to the upper-bound part in the definition of Gamma-convergence, with an added precision needed for statements in the finite temperature case. 

\begin{theo}[Upper bound construction.] \label{thuppb} Let the potential $V$ satisfy assumptions \eqref{assumpV1}--\eqref{assumpV4}.
Assume $P\in \P (X) $ is admissible. 

Then, for any  $\eta>  0$, there exists $\delta >0$ and for any $n$ a subset $A_n\subset\mr^n$ such that $|A_n|\ge n!(\delta/n)^n$ and for every sequence $\{\nu_n= \sum_{i=1}^n \delta_{y_i}\}_n$ with $(y_1,\dots,y_n)\in A_n$ the following holds.

i) We have the upper bound
\ed{
\begin{equation}\label{bsw}\limsup_{n \to \infty}\widehat{F_n} (\nu_n) \le  \tW(P)  +\eta.\end{equation}}

ii) There exists $\{\j_n\}_n$ in $L^p_\loc(\mr^2,\mr^2)$ such that $\div \j_n = 2\pi( \nu_n' -\bm'\dr)$ and such that  the image $P_n$ of $dx_{|\E}/|\E|$ by the map $x\mapsto \(x,\j_n(n x+\cdot)\)$ is such that
\begin{equation}\label{convpn} \limsup_{n \to \infty} \dist(P_n,P) \le \eta, \end{equation}
where $\dist$ is a distance which metrizes the topology of weak convergence on $\P(X)$.
\end{theo}

\begin{remark}
Theorem~\ref{thuppb} is only a partial converse to Theorem~\ref{thlowb}  because the constructed $\j_n$ need not be a gradient,  hence in general \smallskip  $\j_n\neq \j_{\nu_n}$.
\end{remark}

A direct consequence of Theorem~\ref{thuppb} (by choosing $\eta=1/k$ and applying a diagonal extraction argument)  is 
\begin{coro}\label{corouppb} Under the same hypothesis as Theorem~\ref{thuppb}  there exists a sequence $\{\nu_n= \sum_{i=1}^n \delta_{x_i}\}_n$ such that 
\begin{equation}\label{thup}\limsup_{n \to \infty} F_n(\nu_n)  \le \tW(P). \end{equation}
Moreover there exists  a sequence $\{\j_n\}_n$ in $L^p_\loc(\mr^2,\mr^2)$ such that $\div  \j_n = 2\pi(\nu_n' -\bm'\dr)$ and  such that defining $P_n$ as in \eqref{pnun}, with $\j_n$ replacing $\j_{\nu_n}$, we have $P_n\to P$ as $n \to \infty$.
\end{coro}
}

\subsection{Main results}\label{sec1.3}

\edd{
Theorems~\ref{thlowb} and ~\ref{thuppb} have straightforward and not-so straightforward consequences which form our main results.

\begin{theo}[Microscopic behavior of weighted Fekete sets] \label{th2}  Let the potential $V$ satisfy assumptions \eqref{assumpV1}--\eqref{assumpV4}.
If $(x_1,\dots, x_n)$ minimizes $w_n$ for every $n$ and  $\nu_n=\sum_{i=1}^n \delta_{x_i}$, then  $P_{\nu_n}$ as defined in \eqref{pnun} converges as $n\to \infty$ to  
 $$P_0=\frac{dx_{|\E}}{|\E|} \otimes  P_{m_0(x) \mathbb{Z}} $$
and 
 $$ \lim_{n \to \infty} F_n(\nu_n) =\lim_{n\to \infty} \widehat{F_n} (\nu_n)=
 \min \tW,
 \quad  \lim_{n\to \infty}\sum_{i=1}^n \zeta(x_i) = 0.$$
\end{theo}

\begin{proof} This follows from the comparison of Theorem~\ref{thlowb} and Corollary~\ref{corouppb}, together with \eqref{fnhat}: For minimizers, \eqref{thlow} and \eqref{thup} must be equalities. Morever we must have $\lim_n( F_n(\nu_n) -  \widehat{F_n} (\nu_n)) = 0$ --- that is $ \lim_{n\to \infty}\sum_{i=1}^n \zeta(x_i) = 0$ --- and $P$ must minimize $\tW$, hence be equal to $P_0$ in view of  Corollary~\ref{leble}. By uniqueness of the limit, the statement is true without extraction of a subsequence.
\end{proof}

}
It can be expected that $\zeta$ (which is positive exactly in the complement of $\E$) controls the distance to $\E$ to some power.  One can show this under suitable assumptions on $V$ by observing that $U^{\mu_0}$ as in \eqref{optcondmo} is the solution to a fractional obstacle problem and using the results in \cite{crs}.
  \medskip

\ed{We next turn to the situation with temperature. The estimates on  $w_n$ that we just obtained first  allow to deduce, as  announced, a  next order asymptotic expansion of the partition function, which becomes sharp as $\beta \to \infty$.
\begin{theo} \label{valz}Let $V$ satisfy assumptions \eqref{assumpV1}---\eqref{assumpVsupp}. There exist functions $f_1, f_2$ depending only on $V$, such that for  any $\beta_0>0 $ and any $ \beta\ge \beta_0$, and for $n$ larger than some $n_0$ depending on $\beta_0$, we have
\begin{equation}
n\beta f_1(\beta) \le \log \Z -\(- \frac{\beta}{2} n^2 \I(\mo) + \frac{\beta}{2} n \log n\)\le  n \beta f_2(\beta),\end{equation}with $f_1,f_2$ bounded in $[\beta_0,+\infty)$ and
\begin{equation}
\lim_{\beta\to\infty} f_1(\beta)=\lim_{\beta\to\infty}f_2(\beta)=\frac{ \min \tW}{2}.
 \end{equation}
\end{theo}
\begin{remark} In fact we prove that the statement holds with  $f_2(\beta)=\frac{\min \tW}{2} + \frac{C}{\beta}$ for any $C>\log |\E|$.\end{remark}
As mentioned above, this result  can be compared to the expansions known in the literature, which can also be obtained as soon as a Central Limit Theorem is proven for general enough $V$'s, cf.  
\cite{joha,bg1,bg2,sh1,sh2}. These previous results generally assume more regularity on $V$ though. It is also not obvious to check that the formulas agree when $\beta \to \infty$ (for which $\min \tW$ is completely explicit, cf. \eqref{defa}) because the coefficients in these prior works are in principle computable but in quite an indirect manner. 

Our method also allows to give a statement on the thermal states themselves (the complete statement in the paper can be phrased as a next order large deviations type estimate,   to be compared to Theorem \ref{thbg}.)

\begin{theo}\label{th4} Let $V$ satisfy \eqref{assumpV1}--\eqref{assumpVsupp}.  There exists $C_\beta>0$ such that $\lim_{\beta\to \infty}C_\beta=0$  and such that the following holds. 
If $\beta>0$ is finite, the law of $P_{\nu_n}$, i.e. the push-forward of $\Q$ by $i_n$ defined in \eqref{in} converges weakly, up to extraction,  to a probability measure $\widetilde{\mathbb{P}^\beta}$  in $\P(\P(X))$ concentrated on admissible probabilities;  and for $\widetilde{\mathbb{P}^\beta}$-almost every $P$ it holds that 
$$\tW (P) \le \min \tW + C_\beta.$$
\end{theo}The first statement in the result is the existence  of a limiting random electric field process, hence equivalently, via projecting by $(x, \j) \mapsto \frac{1}{2\pi} \div \j + m_0(x)\dr$, of a limiting random point process.
The second statement allows to quantify the average distance to the crystalline state  as $\beta$ gets large, using \eqref{estlebl}:
\begin{coro}
\edd{Let $P \in \P(X)$ be admissible, and let us write its disintegration 
$P=  \{ P^x\}_{x\in\E} $ where $x$-a.e. in $\Sigma$, $P^x$ is a probability measure on $\Lp$ concentrated on $\mathcal{A}_{m_0(x)}$.  Let $\ro_2^x$ be the two-point correlation function  of the point process given as the push forward of $P^x$ by $E \mapsto \frac{1}{2\pi} \div E 
+m_0(x) \delta_\mr$.  Let $\ro_{2, \frac1m\mz}$ be the two-point  correlation function associated  to $P_{\frac1m\mz}$ as above. Then, for any  $\widetilde{\mathbb{P}^\beta}$ obtained by Theorem \ref{th4}, it holds for  $\widetilde{\mathbb{P}^\beta}$-almost every $P$ and any smooth compactly supported $\varphi$ that 
$$ \dashint_{\Sigma} \left|\int (\ro_2^x-\ro_{2, \frac1{m_0(x)} \mz} ) \varphi \right|\, dx  \le C_\varphi C_\beta,$$
}
 where $C_\varphi $ depends only on $\varphi$, and $C_\beta $ is as in Theorem \ref{th4}.
\end{coro}

Since $C_\beta \to 0$, our results can thus be seen as a result of crystallization as $\beta\to \infty$. We believe that when $\beta$ is finite a complete large deviations principle should hold with a rate function involving both $\tW$ and a relative entropy term, whose weight decreases as $\beta \to \infty$. This is work in progress \cite{lebles}.}

\smallskip

Finally, let us mention that our method yields estimates on the probability of some rare events, typically the probability that the number of points in a microscopic interval deviates from the number given by $\mo$. We present them below, even though stronger results are obtained in \cite{bey1,bey2}. The results below follow easily  from  the estimate (provided by Theorem \ref{valz}) that
$\widehat{\F}\le C$  except on a set of small probability.


\begin{theo} \label{th3}
Let $V$ satisfy assumptions  \eqref{assumpV1}--\eqref{assumpVsupp}.
There exists a universal constant $R_0>0$ and $c,C>0$ depending only on $V$ such that: For any $\beta_0>0$, any $\beta  \ge \beta_0$, any $n$ large enough depending on $\beta_0$,  for   any $x_1,\ldots,x_n\in \mr$, any $R>R_0$,  any interval $I\subset \mr $ of length $R/n$,  and any $\eta>0$, letting  $\nu_n=\sum_{i=1}^n \delta_{x_i}$,  we have the following:
\begin{equation}\label{local} \log \Q \( \left|\nu_n(I)- n\mo(I)\right|\ge \eta R   \) \le
- c\beta \min(\eta^2,\eta^3) R^2 + C \beta (R+ n) + C  n ,\end{equation}
\begin{equation}\label{contrsob}
\log \Q\( ( 1+ R^2/n^2)^{\hal - \frac{1}{q}}  \|\nu_n- n\mu_0\|_{W^{-1,q}(I)} \ge \eta \sqrt{n} \)
\le - cn \beta \eta^2   +  Cn (\beta + 1)
\end{equation} where $W^{-1,q}(I)$ is the dual of the Sobolev space  $W^{1,q'}_0(I)$, with $1/q+ 1/q'=1$, in particular $W^{-1,1}$ is the dual of Lipschitz functions;
and
\begin{equation}\label{controlez}
\log\Q\( \int\zeta
 \, d\nu_n \ge \eta\) \le - \hal  n \beta\eta+ C  n (\beta+1) .\end{equation}

\end{theo}
%
Note that in these results $R$ can be taken to depend on $n$.

\eqref{local} tells us that the
density of eigenvalues is correctly approximated by the limiting law
$\mo$ at all small scales bigger than $C n^{-1/2}$ for some $C$.
However this in fact should hold at all scales with $R\gg 1$, cf. \smallskip \cite{esy,bey1,bey2}. \eqref{controlez} serves to control the probability that points are outside $\E$ (since $\{\zeta >0\}= \E^c$).
\medskip

\edd{The rest of the paper is organized as follows. In Section 2 we prove Theorem~\ref{thlowb}, in Section~\ref{secupp} we prove Theorem~\ref{thuppb}. In section~4 we prove the remaining theorems. In the appendix we prove  Lemmas~\ref{depoint} and~\ref{lemsplit}, as well as   the main screening result Proposition~\ref{tronque}.}
\smallskip

{\it Acknowledgements:}
We are grateful to  Alexei Borodin,  G\'erard Ben Arous, Amir Dembo, Percy Deift, Nicolas Fournier,  Alice Guionnet and Ofer Zeitouni for their interest and helpful discussions.
E. S. was supported by the Institut Universitaire de France and S.S. by a EURYI award.

\edd{\section{Lower bound} In this section we prove Theorem~\ref{thlowb}.}
\label{seclow}
\subsection{Preliminaries: a mass displacement result}
\edd{In this subsection, we state  the analogue in 1D of Proposition 4.9  in \cite{ss}, a result we will need later. The proposition below asserts that, even though the energy density $\hal |\j|^2 + \pi  \log \eta \sum_p \delta_p$  associated to $W$  is not bounded below, there exists a replacement $g$  which is. The sense in which $g$ is a replacement for the energy density of $W$ (specified in the statement of the proposition)  is what is needed to make the energy density of $W$ effectively behave as if it were bounded from below.}

The density $g$ is obtained by displacing the negative part of the energy-density into the positive part. The proof is identical to that of  \cite{ss} once the one-dimensional setting has been embedded into the two-dimensional one as stated. What follows will be applied to $\nu_n'$, i.e. the measure in blown-up coordinates.

\begin{pro}\label{wspread} Assume $(\nu,\j)$ are such that $\nu = 2\pi \sum_{p\in\Lambda}
  \delta_p$ for some finite subset $\Lambda$ of $\mr$,
   $\div \j = 2\pi (\nu - a(x)\dr)$,   for some $a\in L^\infty(\mr)$, and $\j$ is a gradient. 
   
Then, given $0<\rho<\rho_0$, where $\rho_0$ is universal,  there exists a measure density  $g$ in $\mr^2$  such that
\begin{itemize}
\item[i)]
There exists a family of disjoint closed balls $\mathcal{B}_\ro$ centered on the real line,  covering $\supp(\nu)$, such that the   sum of  the radii  of the balls in $\mathcal{B}_\ro$ intersected with any segment  of $\mr$ of length $1$ is  bounded by $\ro$  and  such that
\begin{equation}\label{lbg}
g\ge - C (\|a\|_{L^\infty}+1) + \frac14 |\j|^2\indic_{ \mr^2 \sm \mathcal{B}_\ro}\quad \text{in} \ \mr^2,\end{equation}
  where $C$ depends only on $\rho$.
 \item[ii)]$$ g = \hal |\j|^2   \text{\ in the complement of }  \ \mr \times [-1,1].$$
\item[iii)]  For any function $\chi$ compactly supported in $\mr$ we have, letting $\bar\chi(x,y) = \chi(x)$,
\begin{equation}\label{wg}\left|W(\j,\chi) - \int \bar{\chi}\,dg\right|
\le C N (\log N + \|a\|_\infty)\|\nab\chi\|_\infty,\end{equation}
where  $N=\#\{p\in\Lambda\mid B(p,\lambda)\cap\supp(\nabla\bar{\chi})\neq\varnothing\}$ and $\lambda$   depends only on $\rho$.\edd{ (Here $\#A$ denotes the cardinality of $A$).}
\end{itemize}
\end{pro}

 Proposition~4.9 of \cite{ss} of which the above proposition is a restatement, was stated for a fixed universal $\rho_0$, but we may use instead in its proof any $0<\rho<\rho_0$, which makes the constant $C$ above depend on $\rho$. Another fact which is true from the proof of  \cite[Proposition~4.9]{ss} but not stated in the proposition itself is that in fact $g = \hal|\j|^2$ outside $\cup_p B(p,r)$ for some constant $r>0$ depending only on $\rho$, and  if $\rho$ is taken small enough, then we may take $r=1$, which yields item ii) of Proposition~\ref{wspread}.

The next lemma shows that a control on $W$ implies a corresponding control on $\int g$ and of $\int |\j|^2$ away from the real axis, growing only like $R$.
\begin{lem}\label{wtog}
Assume that $G\subset\mathcal A_1$ is such that, writing $\nu= \frac{1}{2\pi}  \div \j + \dr $, 
\begin{equation} \label{hypunif}  \forall R>1,\  \frac{\nu(I_R)}{R} < C,\end{equation}
\begin{equation}\label{cvwunif}
 \lim_{R\to +\infty} \frac{W(\j,\chi_R)}{R}  = W(\j) <C,\end{equation}
 hold uniformly w.r.t. $\j\in G$. Then for any $\j\in G$, for every $R$ large enough depending on $G$, we have
\begin{equation}\label{contrnu}
|\nu(I_R)-R|\le C_1 R^{3/4} \log R,
\end{equation}
\begin{equation}\label{jband}
\int_{ I_R\times \{|y|>1\}} |\j|^2 \le C R \(W(\j)+1\),\end{equation}
and  denoting by $g$ the result of applying Proposition~\ref{wspread} to $\j$ for some fixed value $\rho<1/8$,  we have
\begin{equation}\label{chirg}
W(\j, \chi_R ) - C_1 R^{3/4} \log^2 R \le \int_{{I}_R \times \mr} dg \le W(\j, \chi_{R+1}) + C_1 R^{3/4} \log^2 R,\end{equation}
where $\chi_R$ satisfies \eqref{defchi}, $C_1$ depends only on $G$ and $C$ is a universal constant.
\end{lem}
\begin{proof}
We denote by $C_1$ any constant depending only on $G$, and by $C$ any universal constant.
From \eqref{hypunif}, \eqref{cvwunif} we have for any $\j\in G$ that $\nu(I_R) \le C_1 R$  and $W(\j, \chi_R) \le C_1 R$ if $R$ is large enough depending on $G$. Thus, applying \eqref{wg} we have
$$\left|\int {\chi}_R \, dg \right|\le C_1 R (\log R+1)$$
which, in view of the fact that $\chi_R=1$ in $I_{R-1}$ and $g$ is positive outside $\mr \times [-1, 1]$ and bounded below by a constant otherwise, yields that for every $R$ large enough,
\begin{equation}
\label{gir}
\int_{{I}_{R-1} \times\mr} dg\le C_1 R(\log R+1).\end{equation}
This in turn  implies --- using   \eqref{lbg} and the fact that $\hal|\j|^2 = g$ outside $\mr\times[-1,1]$ ---  the first (unsufficient) control
$$\int_{\{(x, y) | (x,0)\notin \cup \B_\rho\}} |\j|^2 \le C_1 R(\log R+1).$$
Since the  sum of  the radii  of the balls in $\mathcal{B}_\ro$ intersected with any segment  of $\mr$ of length $1$ bounded by $\ro<1/8 $,  we deduce by a mean value argument with respect to the variable $x$ that there exists  $t \in [0,1]$ such that
\begin{equation}\label{jcote} \int_\mr  \left|\j\(-\frac{R}{2} -t,y\)\right|^2 +  \left|\j\(\frac{R}{2} +t,y\)\right|^2\,dy   \le C_1 R (\log R+1).\end{equation}
Using now a mean value argument with respect to $y$,  we deduce from \eqref{gir} the existence of  $y_R \in [1, 1+ \sqrt{R}] $ such that
\begin{equation}\label{jhaut} \int_{-\frac{R}{2}-t}^{\frac{R}{2}+t}|\j(x,y_R)|^2+|\j(x,-y_R)|^2\,dx \le C_1 \sqrt{R} (\log R+1).
\end{equation}
Next, we integrate  $\div \j= 2\pi(\nu- \dr)$ on the square $[-\frac{R}{2}-t, \frac{R}{2}+t]\times [-y_R, y_R]$. We find using the symmetry property of Corollary~\ref{sym} that
$$\left|\nu(I_{R-\frac{t}{2}} ) - R +2t \right| \le
\int_{-y_R}^{y_R} \left|\j\(-\frac{R}{2} -t,y\)\right| +  \left|\j\(\frac{R}{2} +t,y\)\right|\,dy  + \int_{-\frac{R}{2}-t}^{\frac{R}{2}+t}|\j(x,y_R)|\,dx.$$
Using the Cauchy-Schwarz inequality and \eqref{jcote}-\eqref{jhaut}, this leads for $R$ large enough to
\begin{equation}\label{rtq}\left| \nu(I_{R-\frac{t}{2}} ) - R \right|\le 2 +  C_1 R^{3/4} \sqrt{\log R+1} +C_1 \sqrt{y_R} \sqrt{R(\log R+1)} \le C_1 R^{3/4} (\log R +1),\end{equation}
and then --- since $\nu(I_R) \ge  \nu(I_{R-t/2})$ --- to
$$\nu (I_R) - R\ge C_1 R^{3/4} \log R.$$

To prove the same upper bound for $R-\nu(I_R)$ we proceed in the same way, but using a mean value argument  to find some $t\in (-1,0)$ instead of $(0,1)$ such that \eqref{jcote} holds, and then \eqref{jhaut} also. We deduce as above that \eqref{rtq} holds and conclude by noting that since $t\in(-1,0)$ we have $\nu(I_R) \le  \nu(I_{R-t/2})$.
This establishes \eqref{contrnu}.

 We may bootstrap this information:   Indeed \eqref{contrnu} implies in particular that $\nu(I_R)- \nu(I_{R-1})  \le C_1 R^{3/4} \log R$ and thus we deduce from \eqref{wg} that \eqref{chirg} holds:
$$\left| W(\j, \chi_R)- \int \chi_R \, dg\right|\le C_1 R^{3/4} \log^2 R.$$
Then since $W(\j, \chi_R)/R\to W(\j)$ as $R\to\infty$ uniformly w.r.t. $\j\in G$ and since $g$ is both bounded from below by a universal constant and equal to $\hal|\j|^2$ outside $\mr\times[-1,1]$, we deduce \eqref{jband}, for $R$ large enough depending on $G$.

\end{proof}

\begin{defi}\label{defG} Assume  $\nu_n = \sum_{i=1}^n\delta_{x_i}$. Letting $\nu'_n
 = \sum_{i=1}^n\delta_{x_i'}$ be the measure in blown-up coordinates, i.e. $x_i' = n x_i$,  and
  $\j_{\nu_n}
 = - \nab H'_n$, where $H'_n$ is defined by \eqref{Hp}, we denote by
 $g_{\nu_n}$
  the result of applying
   Proposition~\ref{wspread}  to $(\nu'_n,\j_{\nu_n})$.
\end{defi}

\edd{\subsection{Proof of Theorem~\ref{thlowb}}}

We start with a result that shows how $\widehat{F_n}$ controls the fluctuation $\nu_n - n\mu_0$.

\begin{lem}\label{lemcontrole}
Let $\nu_n=\sum_{i=1}^n \delta_{x_i}$. For any  interval $I$ of width $R$  (possibly depending on $n$) and
any    $1<q<2$, we have
$$\|\nu_n- n\mu_0\|_{W^{-1,q}(I)}\le C_q (1+R^2)^{\frac{1}{q}- \frac{1}{2}} n^{\hal} \( \widehat{F_n}(\nu_n) + 1\)^\hal.$$
 Here $W^{-1,q}$ is the dual of the Sobolev space $W^{1,q'}_0$ with $1/q+1/q'=1$.
\end{lem}
 \begin{proof}
In   \cite[Lemma 5.1]{ma2d}, we have the following statement
$$\|\nu_n- n\mu_0\dr\|_{W^{-1,q}(B_R)}\le C_q (1+R^2)^{\frac{1}{q}- \frac{1}{2}} n^{\hal} \( \widehat{F_n}(\nu_n) + 1\)^\hal.$$
The proof is based on \cite{STi} which works in our one-dimensional  context as well, thus the proof can be reproduced without change.
 It is immediate to deduce the result. \end{proof}

We now turn to bounding from below $\widehat{F_n}$. The proof is the same as in  \cite[Sec. 6]{ma2d}, itself following the method of \cite{ss} based on the ergodic theorem. We just state the main ingredients.

 Let  $\{\nu_n\}_n$ and $P_{\nu_n}$ be as in the statement of Theorem~\ref{th2}. We need to prove that any subsequence of $\{P_{\nu_n}\}_n$ has a convergent subsequence and that the limit $P$ is admissible and  \eqref{thlow} holds. Note that the fact that the first marginal of $P$ is $dx_{|\E}/|\E|$ follows from the fact that, by definition, this is true of $P_{\nu_n}$.

We thus take a subsequence of $\{P_{\nu_n}\}$ (which we don't  relabel),  which satisfies  $\widehat{\F}(\nu_n)\le C $. This implies that  $\nu_n$ is of the form $\sum_{i=1}^n \delta_{x_{i,n}}$. We let  $\j_n$ denote the electric field  and $\bgn$ the measures associated to $\bnun$ as in Definition~\ref{defG}. As usual, $\bnun'= \sum_{i=1}^n \delta_{ n x_{i,n}}$.

A useful consequence of $\widehat{\F}(\bnun)\le C $ is that, using Lemma \ref{lemcontrole}, we have
\begin{equation}\label{cvnu} \frac1n\bnun\to\mo \quad \text{on} \ \mr.\end{equation}
 We  then set up the framework of Section~6.1 in \cite{ma2d} for obtaining lower bounds on two-scale energies.   We let $G = \E$ and $X= \radon_+\times \Lp \times  \radon$, where $p\in (1,2)$, 
where $\radon_+$ denotes the set of positive Radon measures on $\mr^2$ and $\radon$ the set of those which are bounded below by the  constant $-C_V:= -C (\|\bm\|_\infty +1)$  of Proposition \ref{wspread}, both equipped with the topology of weak convergence.

For $\lambda\in\mr$ and abusing notation we let $\theta_\lambda$ denote both  the translation $x\mapsto x+\lambda$ and the action
$$\theta_\lambda (\nu,\j,g) = \(\theta_\lambda\#\nu, \j\circ\theta_\lambda,\theta_\lambda\# g\).$$
Accordingly the action $T^n$ on $\E \times X$ is defined for $\lambda\in\mr$ by
$$T^n_\lambda (x,\nu,\j,g) = \(x+\frac\lambda{ n}, \theta_\lambda\#\nu, \j\circ\theta_\lambda,\theta_\lambda\# g\).$$
Then we let $\chi$ be a smooth  nonnegative cut-off function with integral $1$ and support in $[-1,1]$ and
define
\begin{equation}\label{deffn} \f_n(x, \nu,\j, g) =\begin{cases} \D\frac{1}{ \pi} \int_{\mr^2} \chi(t)\, dg(t,s)
 & \text{if $(\nu,\j,g) = \theta_{ n x} (\bnun', \bjn, \bgn)$,}\\ +\infty & \text{otherwise.}\end{cases}
\end{equation}
Finally we let,
\begin{equation}\label{deF} \gF_n(\nu,\j, g) =\dashint_\E \f_n\(x,\theta_{x  n}(\nu,\j,g)\)\,dx.
\end{equation}
We have the following relation between $\gF_n$ and $\widehat{\F}$, as $n\to +\infty$ (see \cite[Sec. 6]{ma2d}):
\begin{equation}\label{FF}
\text{ $\gF_n(\nu,\j,g) $ is  }\quad \begin{cases} \le \frac1{|\E|}\widehat{\F}(\bnun) +o(1) & \text{if $(\nu,\j,g) = (\bnun',\bjn,\bgn)$}\\  = +\infty & \text{otherwise}.\end{cases}
\end{equation}
The hypotheses in  Section~6.1 of \cite{ma2d} are satisfied and applying the abstract result, Theorem 6 of \cite{ma2d}, we conclude that
 letting $Q_n$ denote the push-forward of the normalized Lebesgue measure on $\E$ by the map $x\mapsto (x,\theta_{ n x} (\bnun', \bjn,\bgn))$, and $Q = \lim_n Q_n,$ we have
 \begin{equation}\label{thlow1}\liminf_n  \frac1{|\E|}\widehat{\F}(\bnun)\ge \frac{1}{\pi} \int W(\j) \,dQ(x,\nu,\j,g)\end{equation}
and,  $Q$-a.e. $(\j,\nu)\in \mathcal A_{\bm(x)}$.

Now we let $P_n$ (resp. $P$) be the marginal of $Q_n$ (resp. $Q$) with respect to the variables $(x,\j)$. Then  the first marginal of $P$ is the normalized Lebesgue measure on $E$ and $P$-a.e. we have $\j\in\mathcal A_{\bm(x)}$, in particular
$$W(\j)\ge \min_{\mathcal A_{\bm(x)}} W = \bm(x) \(\min_{\mathcal A_1} W  - \pi\log \bm(x)\).$$ Integrating with respect to $P$ and noting that since only $x$ appears on the right-hand side we may replace $P$ by its first marginal there, we find, in view of \eqref{defa} that the lower bound  \eqref{thlow} holds.

\edd{\section{Upper bound} In this section we prove Theorem~\ref{thuppb}.
\label{secupp}
The construction consists of the following.

First  we state our  main screening result, whose proof  is given in the appendix, on which the proof of Theorem~\ref{thuppb} is based, and which is the main difference with the two-dimensional situation. It allows to truncate electric fields to allow all sorts of cutting and pastings necessary for the construction. However, for the truncation process to have good properties, an extra hypothesis (see \eqref{jcvun}) needs to be satisfied.

The second step consists in selecting a finite set of vector fields  $J_1, \dots, J_{N}$ ($N$ will depend on $\ep$) such that the marginal  of the probability $P(x,\j)$ with respect to  $\j$ is well-approximated by measures supported on  the orbits  of the  $J_i$'s under translations. This is possible because $P$ is assumed to be $T_{\lambda(x)}$-invariant.  It is  during this approximation process that we manage to select the $J_i$'s as belonging to a part of the support of $P$ of almost full measure  for which  the extra assumption \eqref{jcvun} holds and the screening can be performed.

Third, we work in blown-up coordinates and  split the region $\E'$ (of  order $n$ size) into many  intervals, and then select the proportion of the intervals that corresponds to the relative weight that the orbit of each $J_i $ carries in the approximation of $P$. In these rectangles we paste a (translated) copy of \ed{(the screened version of)} $J_i$ at the appropriate scale (approximating the density $\bm'$ by a piecewise constant one and controlling errors).

To conclude the proof of Theorem~\ref{thuppb}, we collect all of the estimates on the constructed vector field to show that its energy $\w$ is bounded above in terms of $\tW $ and that the probability measures associated to the construction have remained close to $P$.

In what follows we use the notation $\theta_\lambda\j(x,y) = \j(x+\lambda,y)$ for the translates of $\j$, and $\sigma_m\j(x,y) = m  \j(m x, my)$ for the dilates of $\j$.
}

\subsection{The main screening result}\label{sec3}

 \ed{This result says 
that starting from an electric field with finite $W$ which also satisfies some appropriate decay property away from the real axis, we may truncate it in a  strip of width $R$, keep it unchanged in a slightly narrower rectangle around the real axis, and use the  layer between the two strips to transition to a vector field which is tangent to the boundary, while paying only a negligible energy cost  in the transition layer as $R\to \infty$.  The new electric field $E_R$ thus constructed can then be extended outside of the strip by other vector fields satisfying the same condition of being tangent to the boundary of the strip. Because the divergence of a vector field which is discontinuous across an interface is equal (in the sense of distributions) to the jump of the normal derivative across the interface, pasting two such vector fields together will not create any divergence along the boundary interface. We will thus be able to construct   vector fields that still satisfy globally equations of the form \eqref{eqj}, the only loss being that they may no longer be gradients.
However, this can be overcome by projecting them later onto gradients (in the $L^2$ sense), and since the $L^2$  projection decreases the $L^2$ norm, this operation can only decrease the energy, while keeping the relation \eqref{eqj} unchanged.}

\begin{pro}\label{tronque}
Let $I_R=[-R/2,R/2]$, let $\chi_R$ satisfy \eqref{defchi}.

Assume $G\subset\mathcal A_1$ is  such that there exists $C>0$ such that for any $ \j\in G$  and writing $ \nu = \frac{1}{2\pi} \div \j + \dr$ we have \eqref{hypunif}, \eqref{cvwunif}
and
\begin{equation}\label{jcvun}
 \lim_{y_0 \to + \infty} \lim_{R \to +\infty} \dashint_{I_R} \int_{|y|>y_0} |\j|^2=0,\end{equation}
and such that moreover all the convergences are uniform  w.r.t. $\j\in G$.

Then for every  $0<\ep<1$, there exists  $R_0>0$     such that if $R>R_0$ with  $R\in \mn$, then for every $\j \in G$ there exists a vector field $\j_R\in L^p_{loc}( I_R \times \mr, \mr^2)$
such that the following holds:
\begin{itemize}
\item[i)]   $\j_R \cdot \vec{\nu} =0$ on $\p {I}_R\times \mr$, where $\vec{\nu}$ denotes the outer unit normal.
\item[ii)] There is a discrete subset $\Lambda \subset I_R$ such that $$\div \j_R = 2\pi \(\sum_{p\in\Lambda}\delta_p -  \delta_\mr\) \quad \text{in} \  I_R\times \mr.$$
\item[iii)]  $\j_R(x,y) = \j(x,y)$ for $x \in [- R/2+\ep R, R/2-\ep R]$.
\item[iv)]
\begin{equation}\label{restronque} \frac{W(\j_R,\indic_{I_R\times \mr })}{R} \le W(\j)+ C\ep.\end{equation}
\end{itemize}
\end{pro}

\edd{\begin{remark} The assumption \eqref{jcvun} is a supplementary assumption which allows to perform the screening  but which is {\it not necessarily} satisfied for all $\j\in\mathcal{A}_m$, even those satisfying $W(\j)<+\infty$. We believe a counter example could be constructed as follows: let $z_k = (2^k,0)$ and $$\mu = \sum_{k=1}^\infty \frac{(-1)^k 2^{k/2}}{\sqrt k}(\delta_{z_k} - \delta_{-z_k}),\quad U = \Delta^{-1} \mu.$$ Then 
$$\int_{B(z_k,2^{k-1})\setminus B(z_k,k)} |\nab U|^2\ge \pi\((k-1)\log 2 - \log k\)\frac{2^k}k\ge C_0>0,$$
hence 
$$\dashint_{I_{2^{k+1}}}\int_{|y|>k} |\nab U|^2>C_0,$$
where $C_0$ is independent of $k$. Therefore $\j = \nabla U$ violates \eqref{jcvun}. On the other hand, because the strength of each charge in the sum defining $\mu$ is negligible compared to the distance from the next charge, it is possible to approximate $\mu$ by a measure of the type $\nu - \dr$, where $\nu = \sum_{p\in\Lambda} \delta_p$. Letting $\j = 2\pi \nabla\Delta^{-1} (\nu-\dr)$ would then yield a counter-example.
\end{remark}
}

We have not been able to show that screening is always possible without assuming \eqref{jcvun}. However we will see in  Lemma~\ref{GH} that this assumption is satisfied ``generically" i.e. for a large set of vector-fields in the support of any invariant probability measure, and this will suffice for our purposes.

 \subsection{Abstract preliminaries}
We repeat here the definitions of distances that we used in \cite{ma2d}.
First we choose distances which metrize the topologies of $\Lp$ and $\B (X)$, the set of finite Borel measures on $X=\E\times\Lp$. For $\j_1,\j_2\in\Lp$  we let
$$d_p(\j_1,\j_2) = \sum_{k=1}^\infty 2^{-k} \frac{\|\j_1 - \j_2\|_{L^p(B(0,k)) }}{1+\|\j_1 - \j_2\|_{L^p(B(0,k))}}, $$
and on $X$ we use the product of the Euclidean distance on $\E$ and $d_p$, which we denote $d_X$. On $\B(X)$ we define a distance by choosing a sequence of bounded continuous functions $\{\vp_k\}_k$ which is dense in $C_b(X)$ and we let,  for any $\mu_1,\mu_2\in\B(X)$,
$$ d_{\B}(\mu_1,\mu_2) = \sum_{k=1}^\infty 2^{-k} \frac{|\langle \vp_k,\mu_1-\mu_2\rangle|}{1+|\langle \vp_k,\mu_1-\mu_2\rangle|},$$
where we have used the notation $\langle\vp,\mu\rangle = \int\vp\,d\mu$.

We will use  the following general  facts, whose proofs are in \cite[Sec. 7.1]{ma2d}.

\begin{lem}\label{etavar} For any $\ep>0$ there exists $\eta_0>0$ such that if $P,Q\in\B(X)$ and $\|P-Q\|<\eta_0$, then $d(P,Q)<\ep$. Here $\|P-Q\|$ denotes the total variation of the signed measure $P-Q$, i.e. the supremum of $\langle\vp,P-Q\rangle$ over measurable functions $\vp$ such that $|\vp|\le 1$.
\end{lem}
In particular, if $P = \sum_{i=1}^\infty \alpha_i\delta_{x_i}$ and  $Q = \sum_{i=1}^\infty \beta_i\delta_{x_i}$ with $\sum_i|\alpha_i - \beta_i| <\eta_0$, then  $d_{\B}(P,Q)<\ep$.

\begin{lem}\label{etadir} Let $K\subset X$ be compact. For any $\ep>0$ there exists $\eta_1>0$ such that if $x\in K, y\in X$ and $d_X(x,y)<\eta_1$ then $d_{\B}(\delta_x, \delta_y)<\ep$.
\end{lem}

\begin{lem}\label{etaconv} Let $0<\ep<1$.  If  $\mu$ is a probability measure on a set $A$ and $f,g:A\to X$ are measurable and such that $d_{\B}(\delta_{f(x)}, \delta_{g(x)})<\ep$ for every $x\in A$, then
$$d_{\B}(f\#\mu,g\#\mu) <C \ep(\lep+1)$$ where $\#$ denotes the push-forward of a measure.
\end{lem}

The next lemma shows how, given a translation-invariant probability measure $\tilde{P}$  on $\Lp$, one can  select a good subset $G_\ep$ and  vector fields $J_i$ of  $\Lp$ to approximate it. It is essentially borrowed from \cite{ma2d} except it contains in addition the argument that ensures that we may choose $G_\ep$ to satisfy  the assumption \eqref{jcvun}  needed for the screening.
\begin{lem}\label{GH}
Let $\tilde{P}$ be a translation invariant measure on $X$ such that, $\tilde{P}$-a.e., $\j$ is in $\mathcal A_1$ and satifies
 $W(\j)<+\infty$.
Then, for any $\ep>0$ there exists a compact $G_\ep\subset\Lp$ such that
\begin{itemize}
\item[i)] Letting $0<\eta_0$ be as in Lemma~\ref{etavar} we have
\begin{equation}\label{pgep} \tP(\E\times {G_\ep}^c) <\min({\eta_0}^2,\eta_0 \ep).\end{equation}
\item[ii)] The convergence \eqref{Wroi} is uniform with respect to $\j\in G_\ep$.
\item[iii)] Writing $\div \j = 2\pi(\nu_\j-\dr)$, both $W(\j)$ and $\nu_\j(I_R)/R$ are bounded uniformly with respect to $\j\in G_\ep$ and $R>1$.
\item[iv)] Uniformly with respect to $\j\in G_\ep$ we have
\begin{equation}\label{cvjuniform}\lim_{y_0 \to +\infty} \lim_{R \to +\infty} \dashint_{I_R}\int_{|y|>y_0} |\j|^2 =0\end{equation}
\end{itemize}
Moreover, \eqref{pgep} implies that for any $R>1$ there exists a compact subset $H_\ep\subset G_\ep$ such that
\begin{itemize}
\item[v)] For every $\j\in H_\ep$, there exists $\Lj\subset I_{\wb R}$ such that
\begin{equation}\label{gamj}\text{$|\Lj|< C {R} \eta_0$ and  $\lambda\notin \Lj\implies \theta_\lambda \j\in G_\ep.$}\end{equation}
\item[vi)] We have
 \begin{multline}\label{ppseconde}  d_{\B}(\bar{P},P') < C \ep(\lep+1),\quad\text{where}  \\
P' = \int_{\E\times H_\ep}\frac1{\bm(x)|I_{R}|}  \int_{{\bm(x)}I_{R}\sm\Lj} \delta_x\otimes\delta_{\sigma_{\bm(x)}\theta_\lambda \j}\,d\lambda\,d\tP(x,\j)\\
\qquad\qquad\bar{P} =  \int_{\E\times  \Lp} \delta_x \otimes \delta_{ \sigma_{\bm(x)}\j} \, d\tilde{P} (x,\j)
.\hfill \end{multline}
\item[vii)] $$ \tP( \E \times {H_\ep}^c) < \min (\eta_0, \ep).$$
\end{itemize}
Finally, there exists a partition of $H_\ep$ into $\cup_{i=1}^{N_\ep} H_\ep^i$  satisfying $\diam (H_\ep^i) <\eta_3$, where $\eta_3$ is such that
\begin{equation}\label{eta5}
\j \in H_\ep, \ d_p(\j,\j') <\eta_3 , \
m \in (0, \wb], \ \lambda \in  {\wb} I_{R}\backslash \Lj
    \implies
d_{\B} (\delta_{\sigma_m  \theta_\lambda \j}, \delta_{\sigma_m \theta_\lambda \j'}) <\ep;\end{equation}
 and there exists  for all $i$,  $\j_i \in H_\ep^i$ such that
 \begin{equation}\label{wjk} W(\j_i) < \inf_{H_\ep^i} W +\ep.\end{equation}
\end{lem}
\begin{proof}
The lemma is almost identical to  Lemma~7.6 in \cite{ma2d}, except for item iv). The proof in \cite{ma2d} is as follows: First one proves that there exists $G_\ep$ satisfying items ii) and iii) with $\tilde P(\E\times{G_\ep}^c)$ arbitrarily small, in particular one can choose it so that \eqref{pgep} is satisfied. Then one deduces from \eqref{pgep} the existence, for any $R>1$, of a compact subset $H_\ep\subset G_\ep$ satisfying the remaining properties.
The only difference here is that we must check that there exists $G_\ep$ with $\tilde P(\E\times{G_\ep}^c)$ arbitrarily small satisfying not only items ii) and iii), but iv) as well.   Then, the proof of the existence $H_\ep\subset G_\ep$ satisfying the remaining properties is exactly as in \cite{ma2d}.

Of course, by intersecting sets, it is equivalent to prove that ii), iii), and iv) can be satisfied {\em simultaneously} or {\em separately}, on a set of measure arbitrarily close to full. The proof in \cite{ma2d} shows that this is possible for ii) or iii), it remains to check it for iv). For this we consider $G_n = \{\j\mid W(\j)<n\}$. Then $G_n$ is a translation-invariant set since $W$ is a translation-invariant function, and therefore
by the multiparameter ergodic theorem (as  in \cite{becker}), and since $\tilde{P}$ is translation-invariant, we have
\begin{equation}\label{luria}\int_{\E\times G_n}\( \int_{[-1,1]\times \{|y|>y_0\}}  |\j|^2  \) \, d\tilde{P}(x,\j)  =  \int_{\E\times G_n} \( \lim_{R\to \infty}  \frac{1}{R}\int_{\mr \times \{|y|>y_0\}}\chi_R |\j|^2 \) \, d\tilde{P}(x,\j),\end{equation}
where $\chi_R = \indic_{I_R}\ast \indic_{[-1,1]}$. Then, using Lemma~\ref{wtog} and using the fact that the $g$ there was defined in Proposition~\ref{wspread} hence is equal to $\hal|\j|^2$ on $\mr\times\{|y|>1\}$ we deduce from \eqref{chirg} and the fact that $g$ is bounded below by a constant independent of $\j$ that
$$\lim_{R\to \infty}  \frac{1}{R}\int_{\mr \times \{|y|>y_0\}}\chi_R |\j|^2\le C (1+W(\j)) \le C n$$
holds for every $\j\in G_n$ with $n\ge 1$.

It follows that for every fixed $n\ge 1$ the family of functions
$$\left\{\vp_{y_0}: (x,\j)\mapsto \int_{[-1,1]\times \{|y|>y_0\}}  |\j|^2\right\}_{y_0>1}$$
decreases to $0$ on $\E\times G_n$ as $y_0 \to +\infty$, and is dominated by the bounded, hence $\tilde P$-integrable, function $\vp_1$. Lebesgue's theorem then implies that their integrals on $\E\times G_n$ converge to $0$, hence in view of \eqref{luria} that
$$\int_{\E\times G_n} \( \lim_{R\to \infty}  \frac{1}{R}\int_{\mr \times \{|y|>y_0\}}\chi_R |\j|^2 \) \, d\tilde{P}(x,\j),$$
tends to $0$ as $y_0\to +\infty$. Fatou's lemma then implies that \eqref{cvjuniform} holds for $\tilde P$-almost every $(x,\j)\in \E\times G_n$.

Since $W(\j)<+\infty$ holds for $\tilde P$-a.e. $(x,\j)$, we know that $\tilde P(\E\times G_n)\to 1$ as $n\to +\infty$ therefore the measure of $\E\times G_n$ can be made arbitrarily close to $1$, and then Egoroff's theorem implies that by restricting $G_n$ we can in addition require the  convergence in \eqref{cvjuniform} to be uniform.
\end{proof}

\subsection{Construction}\label{sec4}

In what follows $\E' = n \E$, $\bm'(x) = \bm(x/n)$: we work in blown-up coordinates. In view of assumption \eqref{assumpV2},  we may assume without loss of generality that $\E$ is made of one closed interval $[a, b ]$    (it is then immediate to generalize the construction to the case of a finite union of intervals). In that case $\E'=[na,nb ]$. Let $\bw>0$ be a small parameter. For any integer $n$ we choose real numbers  $a_n$ and $b_n$ (depending on $\bw$) as follows:
Let $a_n$ be the smallest number   and $b_n $ the largest  such that
\begin{align}
\label{q0}
&a_n \ge na  + \frac{n\bw}{\gamma^2}\qquad b_n \le  nb - \frac{n\bw}{\gamma^2}\\
\label{q1}
& \int_{na}^{a_n} \bm'(x)\, dx\in \mn \\
\label{q2}
& \int_{b_n}^{nb} \bm'(x)\, dx\in \mn\\
\label{q3}
& \int_{a_n}^{b_n} \bm'(x)\, dx \in q_\ep \mn\end{align}
 where $q_\ep$ is an  integer, to be chosen later, and $\gamma$ is the constant in \eqref{assumpV3}.
 By \eqref{q0} and assumption  \eqref{assumpV2}, we are sure to have   $\bm'\ge  \bw$ in $\E'_{\bw}:=[a_n, b_n]$.
This fact  also ensures that
 \begin{equation}\label{taille}
|a_n-na|\le  \frac{n\bw}{\gamma^2} + \frac{1}{\bw}\qquad  |b_n-nb|\le  \frac{n\bw}{\gamma^2} + \frac{q_\ep}{\bw}.\end{equation}
We also denote $\E_{\wb} :=\frac{1}{n}\E'_{\wb}$.

Let  $P$  be a probability on $\E\times\Lp$ which  is as in the statement of  Theorem~\ref{thuppb}.
Our goal is to construct a vector field $\j_n$ whose $W$ energy is close  to $\int W\, dP$ and such that the associated $P_n$ (defined as the push-forward of the normalized Lebesgue measure on $\E$ by  $x \mapsto \(x, \j( nx + \cdot)\)$) well approximates $P$.

In $[na, a_n]$ and $[b_n, nb]$,
 we approximate $\bm'(x)\,dx$ by  a sum of Dirac masses at points appropriately spaced, 
and build an associated $\j_n$, whose contribution to the energy will shown to be negligible as $\wb\to 0$. 
We leave this part for the end.

For now we  turn to $[a_n, b_n]$, where  we will do a more sophisticated construction, approaching $P$ via Lemma \ref{GH} and  using Proposition \ref{tronque}.
The idea of the construction is to split the interval $[a_n, b_n]$ into intervals of width  $\sim q_\ep R_\ep$, where $q_\ep$ is an integer  and $R_\ep$ a number, both  chosen large enough,  and then paste in each of these intervals a large number of copies of the (rescaled) truncations of the $J_i$'s provided by Proposition \ref{tronque}, in a proportion following that of $P$.

 \medskip

\noindent
{\it -Step 1: Reduction to a density bounded below}.
We have
$$ P = \int \delta_x\otimes\delta_{\sigma_{\bm(x)}\j}\,dQ(x,\j),\quad\text{where}\quad Q = \int \delta_x\otimes\delta_{\sigma_{1/{\bm(x)}}\j}\,dP(x,\j).$$
Moreover, since the first marginal of $P$ is the normalized Lebesgue measure on $\E$ and since $|\E_\bw|\simeq |\E|$ as $\bw\to 0$, we have
\begin{equation}\label{proba1}
\lim_{\bw\to 0} d_{\B} (P, \bar{P})= 0,\end{equation}
where $\bar P$ is defined by
\begin{equation}\label{proba} \bar{P} = \int \delta_x\otimes\delta_{\sigma_{ \bm(x)} \j}\,d\tP(x,\j),\quad\text{with}\quad
\tP = \int_{\E_{\bw} \times \Lp}  \delta_x\otimes\delta_{\sigma_{1/{\bm(x)}} \j}\,dP(x,\j).
\end{equation}
Clearly $\tP$  is is $T_{\lambda(x)}$-invariant since $P$ is, and in particular it is translation-invariant. In addition, for $\tilde{P}$-a.e. $(x, \j)$, we have $\bm(x)\in[\bw, \wb]$,  a situation similar to \cite{ma2d} where the density was assumed to be bounded below.\medskip

\noindent
{\it -Step 2: Choice of the parameters}.
Let $0<\ep<1$. We define the compact set $G_\ep\subset\Lp$ to be given by  Lemma \ref{GH}. Then, from Proposition~\ref{tronque} applied to $G_\ep$, there exists $R_0>0$  such that for any
integer  $R>R_0$, and any $\j\in G_\ep$, there exists a truncation (in the sense of items i), ii), iii) of Proposition~\ref{tronque}) $\j_R$ satisfying \eqref{restronque}. Applying Lemma~\ref{etadir} on the compact set $\{\sigma_m \j: m\in [\bw,\wb], \j\in G_\ep\}$,  there exists $\eta_1>0$ such that
\begin{equation}\label{eta1}
\text{$m\in  [\bw,\wb]$,  $\j\in G_\ep,\j'\in\Lp$ and $d_p(\j,\j')<\eta_1$} \implies d_{\B}(\delta_{\sigma_m \j},\delta_{\sigma_m \j'})<\ep.\end{equation}

Then we define $R_\ep$ to be such that  $\bw R_\ep > R_0$ and such that for any  $\j,\j'\in\Lp$,
\begin{equation}\label{proxitronque}  \text{$\j = \j'$ on $\bar I_{\bw\ep R_\ep}$}\quad\Longrightarrow\quad d_p(\j,\j')<\eta_1.\end{equation}

Going back to Lemma~\ref{GH}, we deduce the existence of $H_\ep\subset G_\ep$, of $N_\ep\in\mn$ and  of $\{\j_i\}_{1\le i\le N_\ep}$ satisfying \eqref{gamj}, \eqref{ppseconde}, \eqref{eta5} and \eqref{wjk}, with $R$ replaced by $R_\ep$.

Finally, we choose $q_\ep\in\mn$ sufficiently large so that
\begin{equation}\label{qep} \frac {N_\ep}{{q_\ep}} < \eta_0,\qquad \frac{N_\ep}{{q_\ep}^2}\times \max_{\substack{0\le i\le N_\ep\\ m\in[\bw,\wb]}} W(\sigma_m \j_i)< \ep.\end{equation}
\medskip

\noindent
{\it - Step 3:  construction in $[a_n,b_n]$.}  We start by splitting this interval into subintervals with integer ``charge". This is done
by induction by letting $t_0=a_n$ and,  $t_k$ being given, letting $t_{k+1}$ be the smallest $t\ge t_k +q_\ep  R_\ep$ such that $\int_{t_{k-1}}^{t_k} \bm'(x)\, dx \in q_\ep\mn$.  By \eqref{q3} there exists $K\in \mn$ such that $t_K=b_n$, and
\begin{equation}\label{K}K \le \frac{b_n - a_n}{q_\ep R_\ep}\le \frac{n(b-a)}{q_\ep R_\ep}.\end{equation}
Since $\bm'\ge \bw$ in $[a_n,b_n]$, it is clear that $t_k- (q_\ep R_\ep +t_{k-1}) \le q_\ep \bw^{-1}$.
To summarize and letting $I_k= [t_{k-1}, t_k]$, we thus have   \begin{equation}\label{intik}|I_k| \in [q_\ep R_\ep, q_\ep(R_\ep + \bw^{-1})],\qquad
 \int_{I_k} \bm'(x)\, dx\in q_\ep \mn.\end{equation}
In each $I_k$ we will paste $n_{i,k}$ copies of a rescaled version of $\j_i$, where
$$n_{i,k} = \left[  \frac{q_\ep (b_n-a_n)}{|I_k|} p_{i,k}\right] \qquad  p_{i,k}= \tilde{P}\( \frac{1}{n}I_k \times H_\ep^i\),$$
$[\cdot] $ denoting the integer part of a number. Because the first marginal of $\tP$ is the normalized Lebesgue measure on $\E_\bw$ and since $[a_n,b_n]\subset \E_\bw \subset [a,b]$, and $\cup_k I_k = [a_n,b_n]$, we have that
$$\frac{|I_k|}{n(b-a)} \le \sum_{i=1}^{N_\ep} p_{i,k}\le \frac{|I_k|}{b_n-a_n},$$
and therefore $\sum_{i=1}^{N_\ep} n_{i,k} \le q_\ep.$ Also, using in particular \eqref{qep},
$$\sum_{i,k} \left| \frac{|I_k|}{q_\ep (b_n-a_n)} n_{i,k} - p_{i,k}\right| \le \frac{N_\ep}{q_\ep} < \eta_0.$$

We divide $I_k $ into $q_\ep$ subintervals with disjoint interiors, all having the same width $\in [R_\ep, R_\ep +\bw^{-1}]$. Then  for each $1\le i\le N_\ep$ we let $\mathcal{I}_{i,k}$  denote a family consisting of $n_{i,k}$ of these intervals. This doesn't necessarily exhaust $I_k$ since $ \sum_{i=1}^{N_\ep}n_{i,k} \le q_\ep $ so we let $n_{0,k}= q_\ep- \sum_{i=1}^{N_\ep} n_{i,k}$.

We define $m_k$ to be the average of $\bm'$ over $I_k$. From \eqref{intik} we have $m_k |I_k|\in q_\ep \mn$ hence for each $I\in \mathcal I_{i,k}$ we have $R :=|m_k I|\in\mn$, and  $R \in [m_k R_\ep, m_k (R_\ep + \bw^{-1})]$. We  then apply Proposition \ref{tronque} in $I_R$ to the vector field $\j_i$,  which yields a ``truncated"  vector field $\j_{i,I}$ defined in $I_R$, where $R = |m_k I|$. If $I\in \mathcal{I}_{0,k}$  we apply the same procedure  with an arbitrary current $\j_0\in \mathcal{A}_1$ fixed with respect to all the parameters of the construction.

We then set $$\j_n^{(1)}(x)=\sigma_{1/m_k}  \j_{i,I} (x_I+ \cdot)$$ on each interval $I\in \mathcal{I}_{i,k}$,  where $x_I$ is the center of $I$.
The next step is to rectify the weight in  $\j_n^{(1)}$. For this we let  $\mathcal R_k$ be the square $I_k\times(-|I_k|/2, |I_k|/2)$ and let $H_k$ be the solution to
$$\left\{\begin{aligned}
-\Delta H_k&= 2\pi(\bm' - m_k)  && \text{in} \mathcal R_k\\
\frac{\p H_k}{\p \nu }&= 0 && \text{on} \ \p \mathcal R_k.
\end{aligned}
\right.$$
From Lemma \ref{lem48} applied with  $\vp$ and $m_0$ equal to zero,  and using the fact that $\bm$ is assumed to belong to $ C^{0\hal}$, we have for any $q \in[1,4]$,
 \begin{equation}\label{banh}
\int_{\mathcal{R}_k}  |\nab H_k|^q \le C_{q} |I_k|^2    \|\bm'-m_k\|_{L^\infty(I_k)}^q\le  C_q  |I_k|^2  \|\bm\|_{C^0,\hal}^q n^{-\frac q2}. \end{equation}
We then define
\begin{align*}
& \j_n^{(2)}=   \begin{cases} & \nab H_k \quad \text{ in} \  \mathcal{R}_k\\  & 0 \qquad \text{in} \ \bar{I}_k  \backslash \mathcal{R}_k\end{cases}\\ &  \j_n=\j_n^{(1)} + \j_n^{(2)}\  \text{ in} \  \bar{I_k}.\end{align*}
Using Lemma \ref{lemsum} and \eqref{banh}  we deduce using \eqref{intik}  that
\begin{equation}\label{nrj1}
W(\j_n, \indic_{\bar{I_k}}) \le W(\j_n^{(1)}, \indic_{\bar I_k}) +o_n(1), \quad \text{as} \ n\to \infty,\end{equation}
where  $o_n(1)$ tends to zero as $n\to\infty$ and depends on  $\ep, \bw>0$ but not the interval $I_k$ we are considering.
 Summing \eqref{banh} for $1\le k\le K$ and in view of \eqref{K} we find that for any $q\in [1,4]$
\begin{equation}\label{jn12}
\int_{[a_n,b_n]\times\mr} |\j_n^{(1)}-\j_n|^q \le C_{q, \ep, \bw} n^{1-\frac q2}.\end{equation}

On the other hand, in view of the construction and the result of Proposition \ref{tronque} we have
\begin{equation}\label{nrj2}W(\j_n^{(1)},\indic_{\bar{I}_k} ) \le |I_k|\( \sum_{i=0}^{N_\ep} \frac{n_{i,K}}{{q_\ep}} W(\sigma_{m_K} J_i) +C \ep\).\end{equation}

Then, following the exact same arguments as in \cite[Sec. 7]{ma2d} which we do not reproduce here (the only difference is that the rescaling factors $\sqrt{n}$ there should be replaced by $n$), thanks to \eqref{proxitronque}--\eqref{eta1}--\eqref{qep} we find that we can choose $C_1$ in \eqref{qep} such that
\begin{equation}\label{distp0}
d_\B( \bar{P}, P') <C \ep( \lep +1)
\end{equation}
where $$P'= \frac{1}{|\E'_{\bw}|} \sum_{k =1}^K\int_{I_k} \delta_{x_k}\otimes  \delta_{\theta_{\lambda} \j_n^{(1)}}\, d\lambda$$
and stands for $P^{(6)}$ in \cite[Sec. 7]{ma2d}.
Also, and again as in \cite{ma2d}, since \eqref{jn12} holds, and from Lemma~\ref{etadir}, we may replace $\j_n^{(1)} $ with $\j_n$ at a negligible cost, more precisely for any large enough $n$ we have
\begin{equation}\label{distp}
d_\B( \bar{P}, P'') <C \ep( \lep +1)\end{equation}
where
$$P''=  \frac{1}{|\E'_{\bw}|} \sum_{k =1}^K\int_{I_k} \delta_{x_k}\otimes  \delta_{\theta_{\lambda} \j_n}\, d\lambda.$$

\noindent
{\it - Step 3: construction in $[b_n, nb]$}. The construction in $[na,a_n]$ is exactly the same hence will be omitted.
We claim that there exists $\j_n$  defined in $[b_n,nb] \times \mr$ such that
\begin{equation}\label{cur}
\left\{\begin{aligned}
\div \j_n &= 2\pi (\sum_i\delta_{x_i}-\bm' \delta_\mr ) &\quad&\text{in   $[b_n,nb] \times \mr$}\\
\j_n\cdot \vec{\nu} &=0  &\quad&\text{on $\partial  ([b_n,nb]\times \mr)$}
\end{aligned}\right.
\end{equation}
and
\begin{equation}\label{step3}W(\j_n, \indic_{ [b_n, nb]\times \mr}) \le C n \(\bw +o_n(1)\),\end{equation} where $C$ may depend on $\gamma, \wb$ and $\ep$. To prove this claim,  let $s_0=b_n$ and for every $l\ge 1$, let $s_l$ be the smallest $s\ge s_{l-1}$ such that $\int_{s_{l-1}}^{s_l} \bm'(x)\, dx=1$. Since \eqref{q2} holds,  this terminates at some $s_L= nb$ with $L= \int_{b_n}^{nb}\bm'\le \wb |nb-b_n|$.
We then set $x_l$ to be the middle of $[s_{l-1}, s_l]$. We let $u_l$ be the solution in the square $\mathcal{R}_l:=[s_{l-1}, s_{l}] \times [-\hal (s_l- s_{l-1}), \hal (s_l-s_{l-1})]$
$$
\left\{\begin{aligned}
 -\Delta u_l &= 2\pi(\delta_{x_l} - \bm'\delta_\mr) && \text{in } \ \mathcal{R}_l \\
  \frac{\p u_l}{\p \nu} &=0 &&\text{on} \ \p \mathcal{R}_l.\end{aligned}
\right.
$$
This equation is solvable since, by  construction of the $s_l$'s, the right-hand side has zero integral.  Then for each $l$ we let $\j_n=-\nab u_l$ in $\mathcal{R}_l$, and let $\j_n = 0$ in $[b_n,nb]\times \mr\setminus\cup_l \mathcal R_l$.
Clearly  $\j_n$ satisfies \eqref{cur}.

To estimate the energy of $u_l$ we let $u_l = v_l+w_l$ where, letting $m_l= \dashint_{[s_{l-1},s_l] } \bm'$,
$$\left\{\begin{aligned}
 -\Delta v_l &= 2\pi(\delta_{x_i} - m_l\delta_\mr) && \text{in } \ \mathcal{R}_l \\
  \frac{\p v_l}{\p \nu}&=0 &&\text{on} \ \p \mathcal{R}_l,\end{aligned}\right.$$
$$\left\{\begin{aligned}
 -\Delta w_l &= 2\pi(m_l-\bm')\delta_\mr && \text{in } \ \mathcal{R}_l \\
  \frac{\p w_l}{\p \nu}&=0 &&\text{on} \ \p \mathcal{R}_l,\end{aligned}\right.$$
From Lemma~\ref{lem48} and Lemma~\ref{lemrect2} we find, choosing for instance $q=4$ so that $q\in[1,4]$ and $q'<2$,
$$\int_{\mathcal{R}_l} |\nab w_l|^q \le C  (s_l - s_{l-1})^2 \left\|m_l-\bm'\right\|^q_{L^\infty( [s_{l-1}, s_{l}])  },
$$ and $$
W(v_l,\indic_{\mathcal R_l}) = C - \pi\log m_l,\quad \int_{\mathcal R_l} |\nab v_l|^{q'}\le C {m_l}^{q'-2}.$$
From \eqref{assumpV3} and  Lemma~\ref{lemsum}, since $\j_n= - (\nab v_l + \nab w_l)$ in $\mathcal{R}_l$, we have
\begin{multline}\label{wjnb}W(\j_n, \indic_{\mathcal{R}_l})\le C - \pi\log m_l+ C\left\|m_l-\bm'\right\|_{L^\infty( [s_{l-1}, s_{l}])} m_l^{1-\frac2{q'}} (s_l - s_{l-1})^{\frac2q} +\\
+ C\left\|m_l-\bm'\right\|^2_{L^\infty( [s_{l-1}, s_{l}])} (s_l - s_{l-1})^2.\end{multline}
Using  \eqref{assumpV4},
$$\left\|m_l-\bm'\right\|_{L^\infty( [s_{l-1}, s_{l}])} \le C\|\bm\|_{C^{0,\hal}}\frac{(s_l - s_{l-1})^\hal}{\sqrt n}.$$
Replacing in \eqref{wjnb} and letting $q=4$ we deduce that
$$W(\j_n, \indic_{\mathcal{R}_l})\le C - \pi\log m_l+ C\(\frac{s_l - s_{l-1}}{\sqrt n}+ \frac{(s_l - s_{l-1})^3}{n}\).$$
Then, summing with respect to $l$  --- using the fact that from \eqref{taille} we have $\sum_l |s_{l+1} - s_l| \le C n \wb(1+ o_n(1))$, the fact that the integral over $[s_{l-1}, s_l]$ of $\bm'$ is $1$ and that from \eqref{assumpV3} we have $(s_l - s_{l-1})\le n^{\frac13}$--- we find
$$W(\j_n,  \indic_{ [b_n, nb]\times \mr}) \le   C \(\int_{b_n}^{nb}\bm'(x)  - \bm'(x)\log\bm'(x)\,dx + n o_n(1)\) \le C n (\bw + o_n(1)),$$
since $\bm'  - \bm'\log\bm'$ is bounded by a constant depending only on $\bm$ and using \eqref{taille}. This proves \eqref{step3}

\medskip

\noindent {\it - Step 4: conclusion.}
 Once the construction of $\j_n$ is completed, the proof of Theorem~\ref{thuppb} is essentially identical to  that of  \cite[Proposition~4.1]{ma2d}, which is its 2-dimensional equivalent, except that the scaling factor $\sqrt n$ there must be replaced by $n$. We only sketch the proof below and refer to the specific part of \cite{ma2d} for the details.

The test vector-field $\j_n$ has now been defined on all $[na,nb]\times \mr$. It is extended by $0$ outside, and is easily seen to  satisfy the relation $\div \j_n= 2\pi (\nu_n'-\bm')$ for $\nu_n'= \sum_{i=1} \delta_{x'_i}$, a sum of Dirac masses on the real line.
Combining \eqref{step3} with \eqref{nrj1}, \eqref{nrj2} and \eqref{wjk}, we have
$$W(\j_n,\indic_{\mr^2} ) \le \sum_k |I_k|\( \sum_{i=0}^{N_\ep} \frac{n_{i,K}}{{q_\ep}} W(\sigma_{m_K} \j_i) +C \ep+ o_n(1)\) + C n (\bw + o_n(1)).$$
Letting $n\to \infty$ and then $\bw \to 0$, we see that  the error term on the right-hand side can be made arbitrarily small, say smaller than $C\ep$. On the other hand, the reasoning of  \cite{ma2d}, Step~2 in Paragraph~7.4, shows that
$$\sum_k |I_k|\( \sum_{i=0}^{N_\ep} \frac{n_{i,K}}{{q_\ep}} W(\sigma_{m_K} \j_i)\) \le |\E'|\int W(\j) \, dP(x,\j)+ C n\(\ep + o_n(1)\),$$
so that taking $n$ larger if necessary we obtain
\begin{equation}\label{nrjjn}
\frac{1}{|\E'| } W(\j_n, \indic_{\mr^2}) \le \int W(\j) \, dP(x,\j)+ C\ep.
\end{equation}

Then arguing as in Paragraph~7.4, Step 3 of \cite{ma2d},  letting $(x_1, \cdots , x_n)$ be the rescalings to the original scale of the points $x_i'$ i.e. $x_i = x_i'/n$, we have  for $n$ large enough
$$\limsup_{n\to \infty} \frac{1}{n}\(\w(x_1, \dots, x_n) - n^2 \I(\mu_0) +n\log n\) \le \frac{|\E|}{\pi} \int W(\j) \, dP(x,\j) + C\ep.$$

Also  letting $P_n$ be the push-forward of $\frac{1}{|\E|} dx\mid_\E $ by the map $x \mapsto (x, \j_n(nx + \cdot))$, it is easy to see that $d_\B(P'', P_n) <  C \bw$. In view of \eqref{proba1} and \eqref{distp}, and taking $\bw$ small enough, for any given $\ep>0$,  we can achieve
$$d_\B(P, P_n) < C \ep.$$

This proves that items i) and ii) of Theorem~\ref{thuppb} are satisfied by $(x_1,\dots,x_n)$ and $\j_n$. Then, the perturbation argument of Paragraph~7.4, Step~4 in \cite{ma2d} shows that there exists $\delta>0$ and for each $n$ a subset $A_n\subset\mr^n$ such that $|A_n|\ge n!(\delta/n)^n$ and such that for every  $(y_1,\dots,y_n)\in A_n$ there exists a corresponding $\j_n$ satisfying \eqref{bsw} and \eqref{convpn}. This concludes the proof of  Theorem~\ref{thuppb}.

\edd{\section{Proof of Theorems~\ref{thmini}, \ref{valz}, \ref{th4} and \ref{th3}.}}
\label{secth}
\subsection{Proof of Theorem~\ref{thmini}}

By scaling (cf. \eqref{minalai1}), we reduce to $m=1$. The result relies on the fact that   there exists a
minimizing sequence
for $\min_{\ainfty} W$  consisting of periodic vector-fields:

\begin{pro}
 \label{periodisation} There exists a sequence $\{\j_R\}_{R \in  \mn}$ in $\ainfty$ such that
each $\j_R$ is $2R$-periodic (with respect to the $x$ variable) and
   $$\limsup_{R\to \infty} W(\j_R)\le  \min_{\mathcal{A}_1} W.$$\end{pro}

\begin{proof}

The result of Proposition \ref{periodisation} is a consequence of Proposition~\ref{tronque}.

 First, applying Theorem~\ref{th2}, there exists a translation-invariant measure $P$  on $\E\times\Lp$  such that $P$-a.e. $(x,\j)$ is such that $\j$ minimizes  $W$ over $\mathcal{A}_{\bm(x)}$. Then, taking the push-forward of $P$ under $(x,\j)\mapsto\sigma_{1/\bm(x)} \j$, we obtain a probability $Q$ on $\Lp$ such that $Q$-a.e. $\j$ minimizes $W$ over $\mathcal A_1$.

Applying Lemma~\ref{GH} to $Q$, we find that $Q$-a.e. $\j$ is such that $\j \in \mathcal{A}_1$,  such that \eqref{jcvun} holds,  and such that $W(\j) = \min_{\mathcal{A}_1} W$.  Choosing  such an $\j_0$ and applying Proposition \ref{tronque} to $G=\{\j_0\}$, we find that for any given $\ep>0$ and any integer $R$ large enough depending on $\ep$, there exists  $\j_R$ defined on $I_R\times \mr$ such that  $\j_R \cdot \vec{\nu}=0$ on $\p( I_R\times \mr) $ and $ W(\j_R,\indic_{ I_R\times \mr })< R(W(\j_0)+\ep)$. This $\j_R$ can be extended periodically by letting $\j_R(x+kR,y) = \j_R(x,y)$ for any $k\in\mz$.

From the condition $\j_R \cdot \vec{\nu}=0$ on $\p( I_R\times \mr)$ we find that, letting $\Lambda\subset I_R$ be the locations of the Dirac masses in $\div\j_R$, we have $\div \j_R = 2\pi \( \sum_{p\in \Lambda_R} \delta_p - \delta_\mr\)$, where $\Lambda_R = \Lambda + R\mz$. Moreover
$$W(\j_R)  = \frac{W(\j_R,\indic_{ I_R\times \mr})}{|I_R|}\le W(\j_0)+C\ep.$$

 There remains to make $\j_R $ a gradient. Following the proof of Corollary 4.4 of \cite{ss} we let $\tilde{\j}_R = \j_R + \np f_R$ in ${I}_R\times \mr$ where $f_R$ solves  $-\Delta f_R= \curl \j_R$ in ${I}_R\times \mr$ and $f_R=0$ on $\p ({I}_R\times \mr) $.   Then, $\div \tilde{\j}_R= \div \j_R $ and $\curl \tilde{\j}_R= 0$ in $\bar{I}_R$. We can thus find $H_R$ such that $\tilde{\j}_R= \nab  H_R$ in ${I}_R\times \mr$. It also satisfies $\nab H_R \cdot \vec{\nu}=\tilde{\j}_R \cdot \vec{\nu} = \j_R \cdot\vec{ \nu} + \np f_R \cdot \vec{\nu}=0  $ on $\p ({I}_R\times\mr) $. We may then extend $H_R$ to a periodic function by even reflection, and take the final $\bar{\j}_R$ to be $\nab H_R$. This procedure can only decrease the energy (arguing again as in \cite{ss, ma2d}): we have $W(\tilde{\j}_R, \indic_{{I}_R\times \mr})\le W(\j_R, \indic_{{I}_R\times \mr} )$ since
\begin{multline*} \int_{({I}_R \times \mr) \backslash \cup B(p, \eta)  } |\nab H_R - \np f_R|^2 - \int_{{I}_R\times \mr \backslash \cup B(p, \eta) }
 |\nab H_R|^2  \\ = - 2\int_{({I}_R\times \mr)  \backslash \cup B(p, \eta) }
 \nab H_R \cdot \np f_R + \int_{{I}_R\times \mr \backslash \cup B(p, \eta)}
 |\nab f_R|^2.\end{multline*}It can be checked that  the last two terms on the right-hand side converge as $\eta\to 0$ to the integrals over ${I}_R\times \mr $. Also integrating by parts, using the Jacobian structure and the boundary data, we have  $ \int_{{I}_R\times \mr}  \nab H_R \cdot \np f_R =0 $.
Therefore, letting  $\eta\to 0$  in the above yields
   $$ W(\j_R,\indic_{{I}_R \times \mr}) - W( \nab H_R ,\indic_{ {I}_R\times \mr})  \ge 0.$$
   We deduce that $W(\bar{\j}_R) \le   W(\j_R) \le  \min_{\mathcal{A}_1} W+C\ep$, with $\bar{\j}_R$ a $2R$-periodic (with respect to the variable $x$) test vector field belonging to $\mathcal{A}_1$.
 The result follows by a standard diagonal argument.

\end{proof}

The following proposition  could be proven as in \cite{ss}, however we omit the proof here.
\begin{pro}\label{pro1}
$W: L^p_{loc}(\mr^2, \mr^2) \to \mr \cup \{+\infty\} $, $1<p<2$,  is a Borel function.
$\inf_{\mathcal{A}_1} W$ is achieved and is finite.
\end{pro}

The result of Theorem \ref{thmini} will follow from  Proposition  \ref{periodisation}  combined with the following
\begin{pro}[Minimization in the periodic case]
Let $a_1, \cdots, a_N$ be any points in $[0,N]$ and $\j_{\{a_i\}}$ be the corresponding periodic vector field, as in Lemma \ref{casper}. Then
$$W( \j_{\{a_i\}}) \ge W(\j_{\mz} ) =-\pi \log 2\pi $$
where  $\j_{\mz} $ is the electric field  associated to the perfect lattice $\mz$.
\end{pro}
\begin{proof} \ed{The proof relies on a convexity argument.}
First, $W(\j_{\mz})$ is immediately  computed via \eqref{WNlog}, taking $N=1$.

Let us now  consider  arbitrary points $a_1,\dots, a_N$ in $[0,N]$, and assume $a_1<\dots <a_N$.
 Let us also denote $u_{1,i}= a_{i+1}-a_i $, with the convention $a_{ N+1}=a_1+ N $.  We have $\sum_{i=1}^N u_{1,i}=   N$.
 Similarly, let
$u_{p,i}= a_{i+p}-a_{i}$, with the convention $a_{N+l}=a_l+N$.  We have $\sum_{i=1}^N u_{p,i}= p N$.
By periodicity of $\sin$, we may view the points $a_i$ as living on the circle $\mr/(N\mz)$. When adding the terms in $a_i-a_j$ in the sum of \eqref{WNlog}, we can split it according to the difference  $p=j-i$ but modulo $N$.  This way, there remains
\begin{equation}\label{3.12}
W(\j_{\{a_i\}}
)= -\frac{\pi}{N} \sum_{i\neq j}  \log \left|2\sin \frac{\pi (a_i-a_j)}{N}\right|-\pi \log \frac{2\pi}{N}=-
\frac{2\pi}{N}\sum_{p=1}^{[N/2]} \sum_{i=1}^N   \log \left| 2\sin\frac{ \pi u_{p,i}}{T} \right|-\pi \log \frac{2\pi}{N},\end{equation} where $[\cdot ]$ denotes the integer part.
But the function $\log |2\sin x|$ is stricly concave on $[0,\pi]$. It follows that
$$\frac{1}{N}\sum_{i=1}^N \log \left| 2\sin\frac{ \pi u_{p,i}}{N} \right|\le \log \left|2\sin \( \frac{\pi }{N^2} \sum_{i=1}^N  u_{p,i} \) \right|= \log \left|2\sin\frac{ p\pi}{N}\right|
$$ with equality if and only if all the $u_{p,i}$ are equal.
Inserting into \eqref{3.12} we obtain
\begin{equation}
\label{3.2}
W(\j_{\{a_i\}}) \ge  -2 \pi\sum_{p=1}^{[N/2]} \log \left|2\sin\frac{ p\pi}{N}\right|- \pi \log \frac{2\pi}{N}.
\end{equation}
On the other hand, if we take for the $a_i$'s the points of the lattice $ \mz$ viewing them as $N$-periodic, we have $u_{p,i}=  p$ for all $p, i$, so if we compute $W(\j_{\mz})$ using \eqref{3.12}, we find
$$W(\j_{\mz}) = - \frac{2\pi}{N} \sum_{p=1}^{[N/2]} \sum_{i=1}^N   \log \left| 2\sin\frac{
\pi p }{N} \right|
 - \pi \log \frac{2\pi}{N}=- 2 \pi \sum_{p=1}^{[N/2]}    \log \left| 2\sin\frac{ \pi p }{N} \right|
 - \pi \log \frac{2\pi}{N}  .$$ This is the right-hand side of \eqref{3.2}, so \eqref{3.2} proves that
 $W(\j_{\{a_i\}}) \ge W(\j_{\mz})$ with equality if and only if all the $u_{p,i}$ are equal, which one can easily check implies that $\{a_i\}=\mz$.
 \end{proof}

Combining with Proposition~\ref{periodisation}, this proves Theorem \ref{thmini}.

\subsection{Proof of Theorems~\ref{valz} and \ref{th4}}

We may  cancel out all leading order terms and  rewrite the probability law \eqref{loi} as
\begin{equation}\label{loi2}
d\Q(x_1, \dots, x_n) =\frac{1}{ \K} e^{- \frac{n  \beta}{2} \F(\nu)} \, dx_1 \dots dx_n\end{equation}
where \begin{equation}\label{defK}
\K= \Z e^{ \frac{\beta}{2}( n^2 \I (\mo)- n \log n)      }.\end{equation}
A consequence of Theorem~\ref{thuppb} is, recalling  \eqref{defa} :

\begin{coro}[Lower bound part of Theorem \ref{valz}] \label{coro43} For any \ed{$\beta>0$ there exists $C_\beta>0$ such that $\lim_{\beta\to\infty} C_\beta =0$ and
\begin{equation}\label{lbk}
\liminf_{n\to +\infty} \frac{\log \K}n \ge - \frac{\beta}{2}  (\min \tW +C_\beta)  .\end{equation}}
\end{coro}

\begin{proof}
It is exactly the same as in \cite[Corollary 4.7]{ma2d} but just letting
\begin{equation}\label{sigma} \sigma_{m} \j (y) :=  m\,\j( m y).\end{equation}
\end{proof}

For the upper bound part of Theorem~\ref{valz},  we start with the following lemma, which  has the same proof as in \cite[Lemma 3.5]{ma2d}.
\begin{lem}\label{lemintxi}
Letting $\nu_n$ stand for $\sum_{i=1}^n \delta_{x_i}$ we have, for any constant $\alpha>0$ and uniformly w.r.t. $\beta\ge \beta_0>0$,
 \begin{equation}\label{lex}
\lim_{n \to \infty} \( \int_{\mr^n} e^{-\alpha\beta n \int \zeta \,d\nu_n }  \, dx_1 \dots dx_n \)^\frac{1}{n} = |\E|.\end{equation}\end{lem}

\edd{Then, exactly as in \cite{ma2d}, we integrate \eqref{loi2} (recall that $\nu$ stands for $\sum_i \delta_{x_i}$). We find,
$$ 1  =   \frac{1}{\K} \int_{\mc^n} e^{-\hal \beta n \F(\nu)} \, dx_1 \dots dx_n$$ hence
\begin{equation}\label{logq}
0 =- \frac{\log \K}{n} +\frac1n \log \int_{\mc^n} e^{-\hal \beta n \F(\nu)} \, dx_1 \dots dx_n.\end{equation}
We deduce, since $\widehat{F}_n(\nu) = \F(\nu) - 2\sum_i\zeta(x_i)$, that
\begin{equation}\label{logq2}
0  \le  - \frac{\log \K}{n}  + \frac{1}{n}    \log\( e^{-\hal \beta n \inf \widehat{F}_n} \int_{\mc^n} e^{-\beta n \sum_i\z(x_i) } \, dx_1 \dots dx_n\)  .\end{equation}
The result then follows as in \cite{ma2d} from the above lemma and the lower bound of Theorem~\ref{thlowb} which implies that $\liminf_n\inf \inf \widehat{F}_n\ge \min\tW.$

The proof of Theorem~\ref{th4} is   identical to \cite{ma2d} once Theorems~\ref{thlowb} and \ref{thuppb} are known, except for the replacement of the scaling $\sqrt n$ by $n$ and $|A_n|\ge n! (\pi \delta^2/n)^n$ by $|A_n|\ge n! (\pi \delta/n)^n$.
}
\subsection{Proof of Theorem~\ref{th3}}

The proof relies on the following proposition, whose proof is much shorter than  in \cite{ma2d}, due to the simpler nature of the one-dimensional geometry.
\begin{pro}\label{5.1} Let $\nu_n=\sum_{i=1}^n \delta_{x_i}$, and $g_{\nu_n}$ be as in Definition~\ref{defG}. For any $R>1$, for any $x_0\in \mr$, denoting $$D(x_0,R)=\nu_n \(B_{\frac R{n} }(x_0)\)-n\mo\(B_{\frac R{n}}(x_0)\)$$
 we have
 $$ \int_{B_{2R}(x_0')} \,d g_{\nu_n} \ge -C R + c D(x_0,R)^2 \min\(1, \frac{|D(x_0,R)|}{R}\),$$
where $c>0$ and $C$ depend only on $V$.
\footnote{The condition $R>1$ could be replaced by $R>R_0$ for any $R_0>0$ at the expense of a constant $c$ depending  on $R_0$.}\end{pro}
\begin{proof}

 Two cases
can happen: either $D(x_0,R) \ge 0$ or $D(x_0, R) \le 0$.

We start with the first case.
Let us choose $\tau =\min \(2, 1+\frac{D(x_0,R)}{2R\|\bm\|_{L^\infty}}  \)$ and denote $T= \{r \in
[R, \tau R], B_r(x_0') \cap \mathcal{B}_\ro = \varnothing\}$, where $\mathcal{B}_\ro$ is as in Proposition \ref{wspread}. By
construction of $\mathcal{B}_\ro$ and since $\ro<\hal$,  we have $|T|\ge \hal (\tau-1) R$. We then
follow the method of ``integrating over circles" introduced in
\cite{gl7}: let $\mathcal{C} $ denote $\{x \in B_{\tau
R}(x_0')\backslash B_R(x_0'), |x-x_0'|\notin T\}$.

For any $r \in T$, since $\p B_r(x_0')$ does not
intersect $Supp (\nu_n')$, we have \begin{multline} \int_{\p
B_r(x_0')} \j_{\nu_n} \cdot \nu = \int_{B_r(x_0')} \div \j_{\nu_n}= 2\pi \nu_n'
(B_r(x_0')) -
\int_{B_r(x_0')}\bm\( \frac{x}{n}\) \dr \\
\ge  D(x_0,R) -2(\tau-1) R\|\bm\|_{L^\infty}
\ge \hal D(x_0,R)
\end{multline} by assumption and by choice of $\tau$.
 Moreover, for any $r \in T$, we have, by Cauchy-Schwarz,
$$\int_{\p B_r(x_0')}|\j_{\nu_n}|^2 \ge\frac{1}{2\pi r}
 \( \int_{\p B_r(x_0')} \j_{\nu_n} \cdot \vec{\nu}\)^2 \ge \frac{1}{8\pi r}
  D(x_0,R)^2.$$
Integrating over $T$, using $|T|\ge \hal (\tau-1) R$, we have
$$\int_T \frac{dr}{r} \ge \int_{\tau R - \hal (\tau-1) R}^{\tau R}\frac{dr}{r}= -\log \(1- \frac{\tau-1}{2\tau} \)$$
and thus
$$\int_{B_{\tau R} (x_0')\backslash \B_\rho} |\j_{\nu_n}|^2 \ge  c D(x_0,R)^2 \min \( 1,
\frac{D(x_0,R)}{R\|\bm\|_{L^\infty}} \),$$ for some $c>0$ depending only on $\|\bm\|_{L^\infty} $ hence on $V$.
Inserting
into \eqref{lbg},  we are led to
$$\int_{B_{2R}(x_0')}  g_{\nu_n} \ge - C (\|\bm\|_{L^\infty}+1)R +
  c D(x_0,R)^2 \min \( 1,
\frac{D(x_0,R)}{R\|\bm\|_{L^\infty}} \).
$$
The  case $D(x_0, R) \le 0$ is essentially analogous.

\end{proof}

\edd{We also need 
\begin{lem} For any $\nu_n=\sum_{i=1}^n \delta_{x_i}$, we have
\begin{equation}\label{fg}\widehat{\F}(\nu_n) = \frac{1}{n\pi} \int_{\mr^2} dg_{\nu_n}  \end{equation} where $\widehat{\F}$ is as in \eqref{fnhat}
and $g_{\nu_n}$ is the result of applying Proposition~\ref{wspread} to $\nu_n$.

\end{lem}
\begin{proof} This follows from \eqref{wg} applied to $\chi_{R}$, where $\chi_{R}$ is as in \eqref{defchi}. If $R$ is large enough then $\#\{p\in\supp(\nu)\mid B(p,C)\cap\supp(\nabla\bar{\chi})\neq\varnothing\} = 0$ and therefore \eqref{wg} reads
$$W(\j_{\nu_n},\chi_{R}) = \int \bar{\chi}_{R}\,dg_{\nu_n}.$$
Letting $R\to +\infty$ yields $W(\j_{\nu_n},\indic_{\mr^2}) = \int\,dg_{\nu_n}$ and the result, in view of \eqref{defF}.
\end{proof}
}

We now proceed to the proof of Theorem~\ref{th3}, starting with \eqref{local}. If $R>R_0$ and $|D(x_0',R)|\ge \eta R$ then from Proposition \ref{5.1} and using the fact  ---  from Proposition \ref{wspread} --- that $g_{\nu_n}$  is positive outside $\cup_{i=1}^n B(x_i',C)$ and that $g_{\nu_n}\ge -C$ everywhere, we deduce from   \eqref{fg} and \eqref{defF}, \eqref{fnhat} that
\begin{equation}\label{pointf}\F(\nu_n) \ge  \frac{1}{n} \(- CR + c \min( \eta^2 ,  \eta^3) R^2 \)+ 2\int \zeta\,d\nu_n .\end{equation}
Inserting into \eqref{loi2} we find
$$\Q\( \left| D(x_0',R)\right|\ge \eta R\)
\le\frac{1}{\K}  \exp\(C \beta  R -c \beta\min(\eta^2, \eta^3)  R^2\) \int  e^{-n\beta\int\zeta\,d\nu_n} \, dx_1\dots dx_n.$$
Then, using   the lower bound \eqref{lbk} and Lemma~\ref{lemintxi} we deduce that if $\beta\ge\beta_0$ and $n$ is large enough depending on $\beta_0$ then
$$\log \Q\( \left|D(x_0',R)\right|\ge  \eta R\) \le  -c\beta\min(\eta^2,\eta^3) R^2+ C\beta R + Cn\beta+ Cn , $$
where $c, C>0$ depend only on $V$. Thus \eqref{local} is \smallskip established.

 We next turn to \eqref{controlez}.  Arguing as above, from \eqref{fg} we have $\F(\nu_n) \ge - C  + 2  \int\zeta\,d\nu_n.$
Splitting $2\int \zeta \, d\nu_n $ as $\int \zeta\, d\nu_n+\int \zeta\, d\nu_n$, inserting into \eqref{loi2} and using \eqref{lbk} we are led to
$$\Q\( \int\x \, d \nu_n  \ge \eta\) \le e^{ -\hal n \beta \eta + C n (\beta  +1)} \int e^{-n\beta \int \zeta\,d\nu_n} \, dx_1\dots, dx_n,$$
where $C$ depends only on $V$. Then, using  Lemma \ref{lemintxi} we deduce
\eqref{controlez}.

We finish with \eqref{contrsob}. Inserting the result of  Lemma \ref{lemcontrole} into \eqref{loi2}, we have, if $I$ is an interval of width $R/n$
$$\Q\( \|\nu_n - n\mu_0\|_{W^{-1,q}(I)} \ge C_q \eta \sqrt{n} (1+R^2/n^2)^{\frac{1}{q}-\frac{1}{2}}  \)\le
\frac{1}{\K}  e^{- \hal n \beta \eta } \int  e^{-n\beta\int\zeta\,d\nu} \, dx_1\dots dx_n.$$
Arguing as before and rearranging terms yields \eqref{contrsob}.

This concludes the proof of Theorem~\ref{th3}.

\edd{\section{Appendix}}
\subsection{Proof of Lemma \ref{depoint}}
Assume $\j$ and $\j'$ belong to $\mathcal A_m$ and  satisfy \eqref{eqj} with the same $\nu$. Then $f = \j - \j'$ is divergence-free and curl-free, hence can be seen, identifying $\mr^2$ and $\mc$, as an entire holomorphic function $\sum_{n=0}^\infty a_n z^n$. If we assume that $W(\j)$ and $W(\j')$ are finite, then it follows  from  \cite{STi},~Corollary~1.2 that  the growth of the $L^1$ norms of $\j$ and $\j'$ is no worse than $R^{3/2}\sqrt{\log R}$ hence there exists $C>0$ such that for any $R>2$ we have $\|f\|_{L^1(B_R)}\le C R^{3/2} \sqrt{\log R}$. But by Cauchy's formula we have, for any $R>0$ and $t\in [R, R+1]$
 $$a_n= \frac{1}{2i\pi} \int_{\partial B(0,t)} \frac{f(z)}{z^{n+1}} \, dz= \frac{1}{2i\pi} \int_{R}^{R+1} \int_{\p B(0,t)} \frac{f(z)}{z^{n+1}} \, dt.$$
 It follows with the above that $|a_n|\le CR^{3/2} \sqrt{\log R} R^{-n-1}$ which implies, letting $R \to \infty$ that $a_n=0$ for any $n \ge 1$, thus  $f$ is a constant. This constant must then be zero since both $\j$ and $\j'$ are square integrable on the infinite strips $[a,b]\times [1,+\infty]$.

We note that Lemma \ref{depoint} implies in particular
\begin{coro} \label{sym}Under the same assumptions, if $S(x,y) = (x,-y)$ then $\j\circ S = S\circ \j$. \end{coro}
Indeed, it is easy to check that $\j' = S\circ \j\circ S$ satisfies \eqref{eqj} with the same $\nu$ as $\j$, and obviously $W(\j')<+\infty$, hence $\j'  =\j$.


\subsection{Proof of the splitting formula (Lemma \ref{lemsplit})}
 Let $\nu_n = \sum_{i=1}^n \delta_{x_i}$.
First, letting $\triangle$ denote the diagonal of $\mr \times \mr$, we  may rewrite $\w$ as
$$\w(x_1, \dots, x_n) =\int_{\triangle^c}\,- \log |x-y| \, d\nu_n(x)\, d\nu_n(y) + n\int_\mr V(x)\, d\nu_n(x).$$
Splitting $\nu_n$ as $n \mo + \nu_n - n\mo$ and using the fact that $\mo\times \mu_0(\triangle) = 0$, we obtain
\begin{multline*}
w(x_1, \dots, x_n) = n^2 \I(\mo)+  2n \int U^{\mo}(x)\, d(\nu_n - n \mo)(x)+ n\int V(x)\, d(\nu_n - n \mo)(x)\\+
\int_{\triangle^c}\,- \log |x-y| \, d(\nu_n - n \mo)(x)\, d(\nu_n - n \mo)(y) .
\end{multline*}
Since $U^{\mo} + \frac{V}{2}= c + \z$ and since $\nu_n$ and $n \mo$ have same mass $n$, we have
$$2n \int U^{\mo}(x)\, d(\nu_n - n \mo)(x)+ n\int V(x)\, d(\nu_n- n \mo)(x)=  2 n \int \z \,d(\nu_n - \mo)= 2n \int \z \,d\nu_n,$$
using the fact that  $\z=0$ on  the support  of $\mo$.

In addition, we have that  \begin{equation}\label{enren} \int_{(\mr\times \mr) \sm \triangle}\,- \log |x-y| \, d(\nu_n - n \mo)(x)\, d(\nu_n- n \mo)(y)= \frac{1}{\pi} W(\nab H_n ,\indic_{\mr^2}),\end{equation}
where  we define $H_n= - 2\pi \Delta^{-1}\(\sum_{i=1}^n \delta_{x_i} -n\mo\)$. Indeed, the integral might as well be written as over $\mr^2 \times \mr^2 \backslash \triangle$ with the diagonal in $\mr^2 \times \mr^2$; and then the identity is proven in \cite[Section 2]{ma2d}.   Combining  all the above we find
\begin{equation}\label{avantchvar}w(x_1, \dots, x_n) = n^2 \I(\mo)+  2n \int \z \,d\nu_n +\frac{1}{\pi} W(\nab H_n ,\indic_{\mr^2}).
\end{equation}
But, changing variables, we have
$$\hal\int_{\mr^2 \sm \cup_{i=1}^n B(x_i, \eta) } |\nab H_n|^2 =\hal \int_{\mr^2 \sm \cup_{i=1}^n B(x_i', n \eta)}  |\nab H_n'|^2,$$
and by  adding $\pi n\log\eta$ on both sides and letting $\eta\to 0$ we deduce that $W(\nab H_n , \indic_{\mr^2})= W(-\nab H_n', \indic_{\mr^2} ) -  \pi n \log n$. Together with \eqref{avantchvar} this proves the lemma.


\edd{
\subsection{Proof of the screening result Proposition~\ref{tronque}} Proposition~\ref{tronque} is the main hurdle in the analysis of the 1D log gas, and is specific to the one dimensional case. Our main task is to suitably truncate a field $\j$ on a vertical strip $I_R\times \mr$ so that it can be pasted to other fields, or repeated to yield a periodic field. The constraints are that we wish the modification to be localized near the boundary of the strip, and  the value of the  renormalized energy not to increase much. We also need this truncation procedure to be done for a (compact) set of fields all at once, with uniform estimates over this set. 
}

\subsubsection*{Some preliminary construction lemmas}
The following lemmas serve to estimate the energy of explicit vector fields on boxes, which will be later combined to \ed{build the  test vector fields in the transition strips we use for the screening}. 
 They are adaptations of \cite{ss} and rely on elliptic equations estimates.

\begin{lem}\label{lem48}
Let $\mathcal{K} $ be the  square $[-\frac{L}{2}, \frac{L}{2}]^2$. Let  $ \vp \in L^2 (\p \mathcal{K})$ and
 $a\in L^\infty([-\frac{L}{2}, \frac{L}{2}])$ be such that
  $\int_{-\frac{L}{2}}^{\frac{L}{2}} a(x) \, dx = - \int_{\p\mathcal{K}} \vp$.
  Then, $a_0$ being  the average of $a$ over $[-\frac{L}{2},\frac{L}{2}]$,
   the solution (well defined up to an additive constant) to   
\begin{equation}\label{target1}
\left\{\begin{aligned}
- \Delta u &= a \delta_\mr &&\text{in} \ \mathcal{K}\\
\frac{\p u}{\p \nu} &= \vp &&\text{on} \ \p \mathcal{K}.
\end{aligned}
\right.
\end{equation}
satisfies for every $q \in [1, 4]$
$$\int_{\mathcal{K}} |\nab u|^q \le C_{q} \(a_0^q L^2+ L^2 \left\|a-a_0 \right\|^q_{L^\infty( [-\frac{L}{2}, \frac{L}{2}])  } + L^{2-\frac{q}{2}} \|\vp\|_{L^2(\p \mathcal{K} ) }^q  \).
$$
\end{lem}
\begin{proof}
We write the solution  $u$ of \eqref{target1} as $u=u_1+ u_2+u_3$ where
\begin{equation}\label{targetu1}
\left\{\begin{aligned}
-\Delta u_1 &= a_0 \delta_{\mr} && \text{in} \ \mathcal{K}\\
\frac{\p u_1}{\p \nu } &= \bar{\vp} && \text{on} \ \p \mathcal{K}
\end{aligned}
\right.
\end{equation}
where   $\bar{\vp}$  is equal to  $0$ on the vertical sides of the square and to $\frac{a_0}{2L}$ on both horizontal sides;
\begin{equation}\label{targetu2}
\left\{\begin{aligned}
-\Delta u_2&= (a-a_0)\delta_{\mr} && \text{in} \ \mathcal{K}\\
\frac{\p u_2}{\p \nu }&= 0 && \text{on} \ \p \mathcal{K}
\end{aligned}
\right.
\end{equation}
and
$$
\left\{\begin{aligned}
-\Delta u_3&= 0 && \text{in} \ \mathcal{K}\\
\frac{\p u_3}{\p \nu }&= \vp- \bar{\vp} && \text{on} \ \p \mathcal{K}.
\end{aligned}
\right.
$$
The  solution of \eqref{targetu1} is (up to a constant)
$u_1(x,y)= \frac{a_0}{2}|y|$. Hence
\begin{equation}\label{nrju1}
\int_{\mathcal{K}} |\nab u_1|^q= \(\frac{\bm }{2}\)^q L^2.\end{equation}
For $u_2$, we observe that  $\|(a-a_0)\delta_{\mr}\|_{  W^{-1,q}(\mathcal{K}) }$ is controlled,  for any  $q<\infty$, by $\|a-a_0\|_{L^\infty([-\frac{L}{2}, \frac{L}{2}])}$. Therefore, using elliptic regularity for \eqref{targetu2}, $\|\nab u_2\|_{L^q(\mathcal{K})}$ is controlled by  $\|a-a_0\|_{L^\infty([-\frac{L}{2}, \frac{L}{2}])}$ and a  scaling argument shows that for any $q<\infty$
\begin{equation}\label{nrju3}
\int_{\mathcal{K}}| \nab u_2|^q\le C_q L^2 \|a-a_0\|^q_{L^\infty([-\frac{L}{2}, \frac{L}{2}])}  .\end{equation}
Finally,  in the proof of  Lemma 4.16  of \cite{ss} it is shown that  for any $q\in[1,4]$
\begin{equation}\label{nrju2}
\int_{\mathcal{K}}| \nab u_3|^q\le C_{q} L^{2-\frac{q}{2}} \|\vp\|_{L^2(\p \mathcal{K})}^q.
\end{equation}
Combining \eqref{nrju1}, \eqref{nrju3} and \eqref{nrju2}, we obtain the result.

\end{proof}

\begin{lem}\label{lemrect2} Let $m$ be a positive constant and
let $\mathcal{K}$ be a square  of center $0$, and sidelength $1/m$. Then the solution to
$$
\left\{\begin{aligned}
-\Delta f&= 2\pi( \delta_0
  - m \delta_\mr) && \text{in} \ \mathcal{K}\\
\frac{\p f}{\p \nu }&= 0  && \text{on} \ \p \mathcal{K}
\end{aligned}
\right.
$$
satisfies
\begin{equation}\label{nrjf2}
\lim_{\eta \to 0}\left|\int_{\mathcal{K}\backslash B(0, \eta)}|\nab f|^2
+2 \pi \log \eta\right| = C - \pi\log m
\end{equation} where $C$ is universal,
and for every $1\le q<2$
\begin{equation}\label{nrjfp}
\int_{\mathcal{K}} |\nab f|^q \le C_q m^{q-2},\end{equation}
where $C_q$ depends only on $q$.
\end{lem}

\begin{proof} By scaling we can reduce to the case of $m=1$. Then,
it suffices to observe that $f(z)= - \log |z|+S(z)$ with $S\in W^{1,\infty}(\mathcal{K} )$ and scale back.
\end{proof}

We note that $W$  as defined in \eqref{WR} still makes sense for vector fields satisfying $\div \j= 2\pi (\sum \delta_{x_i}-m)$ \ed{which are not necessary gradients}, as  long as $ \j$ is a gradient  in  $\cup B(x_i, \eta_0) $ for some $\eta_0>0$. This is the notion we will use repeatedly below.

\begin{lem}
\label{lemsum}
Let $\j_1$ and $\j_2$ be two vector fields defined in a rectangle $\mathcal{R}$ of the plane which is symmetric with respect to the real axis, and   satisfying \begin{align}
& \div \j_1= 2\pi (\sum_i \delta_{x_i} -a_1 \delta_\mr)\quad \text{in} \ \mathcal{R}\\
& \div \j_2= a_2\delta_\mr \quad \text{in} \ \mathcal{R}\end{align}
 and $\curl \j_1$ and $\curl \j_2$ vanish near the $x_i$'s,
 for some distinct points $x_i \in \mr$ and some bounded functions on the real line, $a_1$ and $a_2$.
Then, for $q<2$ and $q'$ its conjuguate exponent, we have $\j_1\in L^q(\mathcal{R})$ and $\j_2\in L^{q'}(\mathcal{R})$ and
$$W(\j_1+\j_2, \indic_{\mathcal{R}})\le W(\j_1, \indic_\mathcal{R})   + \|\j_1\|_{L^q(\mathcal{R})}\|\j_2\|_{L^{q'}(\mathcal{R}) } +\hal\|\j_2\|_{L^2(\mathcal{R})}^2,$$ where $W$ is still defined as in \eqref{WR}.

\end{lem}
\begin{proof}We have
$$\int_{\mathcal{R} \backslash \cup_i B(x_i, \eta)}  |\j_1+\j_2|^2 = \int_{\mathcal{R}\backslash \cup_i B(x_i, \eta)}
|\j_1|^2 + |\j_2|^2 + 2\j_1 \cdot \j_2.$$By H\"older's inequality we have \begin{equation*}
\left| \int_{\mathcal{R} \backslash \cup_i B(x_i, \eta)} \j_1 \cdot \j_2 \right| \le \|\j_1\|_{L^q(\mathcal{R}) } \|\j_2\|_{L^{q'}(\mathcal{R})}\end{equation*} The result easily follows.

\end{proof}

\subsubsection*{Proof of Proposition \ref{tronque}}
\ed{We start from a given electric field $E$ in $\mr$ and restrict it to the strip $[-R/2,R/2] \times \mr$. The steps of the screening then go as follows:
\begin{itemize}
\item as a preliminary, we show that with the assumptions placed on the electric field, it  decays fast enough away from the real axis.
\item By a mean value argument, we find a good substrip $[-t, t]\times \mr$ on the 
boundary of which the $L^2$ norm of $E$ is not too large.  This is possible because the energy $W(E)$ which we control is ``morally" equivalent (via the use of the mass displaced density $g$)  to a control on $\dashint_{[-R/2, R/2]\times \mr} |E|^2$. We also want the strip to be only slightly narrower, i.e. $R-t$ small with respect to $R$.
\item We keep the vector field unchanged in $[-t,t] \times [-y_R, y_R]$ unchanged and define a new vector field $E_R$ in the transition strip $(I_R\backslash I_{2t}) \times \mr$, as well as in the parts far from the real axis: $[-R/2,R/2]\times ((-\infty,-y_R]\cup [y_R, +\infty))$.
The new vector field has to satisfy a relation of the form \eqref{eqj} but not necessarily be a gradient, and it has to have the same normal component as $E$ on the boundary of $[-t,t]\times \mr$ so as not to create any  new divergence there. 
This new vector field is constructed by splitting the region in which it needs to be defined into suitable rectangles and semi-infinite strips (cf. Fig \ref{fig}), and constructing it separately in each piece while keeping again the normal components on each interface continuous (so as again not to any create divergence). The construction in each piece is done 
thanks to the preliminary Lemmas \ref{lem48} and \ref{lemrect2} which provide at the same time the appropriate vector fields and estimates on their energy.
\item We check that $y_R$ can be chosen so that the energy of all the combined vector fields does not exceed the original energy in the strip plus a negligible error. 
\end{itemize}

}
First we note that, in view of \eqref{jcvun},
if we assume $G\subset \mathcal A_1$ satisfies the hypothesis of Proposition~\ref{tronque} and  $0<\ep <1$, then
there exists $y_0>0$ and $R_0>0$ such that for all $\j \in G$, we have
\begin{equation}\label{jcvun2}
\forall R>R_0, \qquad \int_{I_R\times \{ |y|>y_0\}} |\j|^2 <\ep^{10} R.\end{equation}
This motivates the following lemma \ed{in which we show an explicit decay of these vector fields away from the real axis}.
 \begin{lem}\label{harmoni}
Let $\j\in \mathcal A_1$  satisfy \eqref{jcvun2}, where $0<\ep<1$. Then, denoting $z=(x,y)$, if $|y|> \max(2y_0,R_0) $, we have
$$|\j|^2(z) \le C  \frac{\ep^{10} (|x|+|y|)}{|y|^2},$$ where $C$ is universal.
\end{lem}
\begin{proof}
Each of the coordinates of $\j$ is harmonic in the half plane $\mr^2_+ = \{y> 0\}$ since $\div \j=\curl \j=0$ there. Therefore $|\j|^2 $ is sub-harmonic. Thus, if $B(z, |y|/2)\subset \mr^2_+$ then by the maximum principle we have
$$|\j|^2 (z)\le \dashint_{B(z, |y|/2)}  |\j|^2. $$
If $y>2y_0$, then $B(z, |y|/2)\subset [x- \frac{|y|}{2}, x+ \frac{|y|}{2}] \times [y_0,+\infty) \subset [-|x|-|y|, |x|+|y|] \times [y_0,+\infty).$ Thus in view of \eqref{jcvun2}, if $|x|+|y|> R_0/2$ we have
$$|\j|^2(z) \le \frac{8}\pi \ep^{10} \frac{|x|+|y|}{y^2},$$
and the result follows, by symmetry with respect to the $x$-axis.\end{proof}

\ed{The next result is about finding the ``good" boundary of a slightly narrower substrip  via a mean-value argument.}

\begin{lem}\label{bonbord}
Let $G$ satisfy the assumptions of Proposition \ref{tronque}. Then for any $\j \in G$, any $0<\ep<1/2$ and any $R$ large enough depending on $G$, $\ep$, we may find  $t \in [\frac{R}{2}- \ep R, \frac{R}{2}- \hal \ep R]$ such that
\begin{equation}\label{cotevert}
\int_{(\{-t\} \cup\{t\})\times \mr }|\j|^2 \le C \ep^{-1}\end{equation}
where $C$ depends only on $G$, and
\begin{equation}\label{438}
\lim_{R\to \infty} \frac{1}{2t} W(\j, \indic_{ I_{2t}\times \mr } ) = W(\j)\end{equation}
uniformly in $G$.
\end{lem}
\begin{proof}
Take $\j\in G$ and apply Proposition~\ref{wspread} to $\j$ for some fixed $0<\rho< 1/8$. We obtain a density $g$ and balls $\B_\rho$. Now, using \eqref{chirg} in Lemma \ref{wtog} together with  the bound \eqref{cvwunif}, we deduce that if $R$ is large enough depending on $G$ then for any $\j\in G$ and denoting $g$ the result of applying Proposition~\ref{wspread} to $\j$ we have
\begin{equation}\label{ixg}\int_{x=R/2-\ep R}^{R/2-\ep R/2}\int_\mr \(g(x,y)+g(-x, y)\)\,dx\,dy\le C R + \int_{I_R\times \mr} \,dg\le C R,\end{equation}
where $C$ depends only on $G$ and we have used the fact that $g\ge -C$ everywhere and $g\ge 0$ on the set $\{|y|>1\}$. Then by using the fact that the radii of the balls in $\B_\rho$ which intersect any given interval of length $1$ is bounded by $1/8$ we deduce that if $R$ is also large enough depending on $\ep$, the measure of the set $A$ of $x\in[R/2-\ep R,R/2-\ep R/2]$  such that $\{x,-x\}\times\mr$ does not intersect $\B_\rho$ is bounded below by $\ep R/4$. This and \eqref{ixg} implies that the set $T$ of  $t\in A$ such that $\int_\mr
\(g(t,y)+g(-t,y)\)\,dy < C/\ep$ has measure at least $\ep R/8$ if $C$ is chosen large enough depending on $G$, and   \eqref{lbg} and the fact that $g = \hal|\j|^2$ outside $\mr\times[-1,1]$ imply that \eqref{cotevert} holds  for $t\in T$. Thus
\begin{equation}\label{atleast} \left|\{t\in [R/2-\ep R,R/2-\ep R/2]\mid \text{\eqref{cotevert} holds}  \}\right| \ge \frac{\ep R}8.\end{equation}

For \eqref{438} we argue as in \cite{ss}, Lemma 4.14.
We let $\chi:[0,+\infty)\to\mr$ be a monotonic  function with compact support and let $\bar\chi(x,y) = \chi(|x|)$. First we note that  for any Radon measure $\mu$ in $\mr^2$  we have
$$\int \bar\chi \, d\mu= - \int_{t=0}^{+\infty} \chi'(t)  \mu( I_{2t}\times \mr)\, dt = - 2\int_{u=0}^{+\infty} \chi'(u/2)  \mu( I_{u}\times \mr)\, du .$$
This implies straightforwardly using the definition of $W(\j,\chi)$ that
$$2\int_{u=0}^{+\infty} \( W(\j, \indic_{I_{u} \times \mr} ) - g(I_{u} \times \mr) \) \, \chi'(u/2)\, dt= \int  \chi \, d g -  W(\j, \chi).$$
On the other hand, by \eqref{wg} and applying Lemma~\ref{wtog}, \eqref{contrnu}, if $\chi'$ is supported in $[x,y]\subset [R/2, R]$, then the right-hand side is bounded by  $ C(|x-y|+ R^{3/4} \log^2 R) \|\chi'\|_\infty$ for any $R$ large enough depending on $G$.  Given now any $\rho:\mr_+\to\mr$ supported in $[x,y]\subset [R/2, R]$ we may consider the positive and negative parts $\rho_+$ and $\rho_-$, and  and their primitives $\chi_+$ and $\chi_-$ with compact support, which are monotonic. Applying the above to $\chi_+$ and $\chi_-$ we find
$$\int_{u=0}^{+\infty} \( W(\j, \indic_{ I_{u} \times \mr} ) - g(I_{u} \times \mr) \) \, \rho(u)\, du\le  C(|x-y|+ R^{3/4} \log^2 R)\|\rho\|_\infty.$$

Since this is true for any  $\rho$  supported in $[R-2\ep R,R-\ep R]$, it follows by duality  that
\begin{equation}\label{bl1}\int_x^y |W(\j, \indic_{ I_{u} \times \mr} ) - g( I_{u} \times \mr) |\, du\le  C(|x-y|+ R^{3/4} \log^2 R).\end{equation}
Then we divide $[R-2\ep R, R-\ep R]$ into, say, $[\sqrt R]$ intervals $I_1,\dots, I_{[\sqrt R]}$ so that their length is equivalent to $\ep\sqrt R$ for large $R$. Then, for large enough $R$, on each such interval \eqref{bl1} implies that
$$\int_{I_k} |W(\j, \indic_{ I_{u} \times \mr} ) - g( I_{u}\times \mr) |\, du\le  C R^{3/4} \log^2 R.$$
Therefore the set of $u\in I_k$ such that $|W(\j, \indic_{ I_{u} \times \mr} ) - g( I_{u} \times \mr ) |\le 8 C R^{3/4} \log^2 R $ has measure at least $7|I_k|/8$. \edd{Since this is true on each $I_k$, and since $\cup_k I_k = [R-2\ep R, R-\ep R]$, letting $u=2t$}  the set of $t \in[R/2-\ep R,R/2-\ep R/2]$ such that $|W(\j, \indic_{ I_{2t} \times \mr} ) - g(I_{2t}\times \mr ) |\le 8 C R^{3/4} \log^2 R$ has measure at least $7\ep R/16$. Together with \eqref{atleast}, this implies the existence of $t\in[R/2-\ep R,R/2-\ep R/2]$ such that both  \eqref{cotevert} and \eqref{438} hold.

\end{proof}

We now prove Proposition~\ref{tronque}. Let $G$ satisfy its hypothesis and choose $0<\ep<1$, and $\j\in G$.  Applying Lemma~\ref{bonbord} we find that if $R$ is large enough depending on $G$, $\ep$, then there exists $t\in [R/2-\ep R,R/2-\ep R/2]$ such that \eqref{cotevert} and \eqref{438} hold. For any such {\em integer} $R\in\mn$ we may also
choose $y_R$  such that
\begin{equation}
\label{yr}
 \ep^{3} R < y_R < \ep^{5/2}  R .\end{equation}
Finally we choose $s>t $ such that $s-t \in [y_R, y_R +1] $ and $\frac{R}{2}-s \in \mn,$ and start constructing the vector field $\j_R$.
\medskip

\noindent
{\it -  Step 1: splitting the strip}.
We split the strip $I_R=[-\frac{R}{2}, \frac{R}{2}]\times \mr$ into several rectangles and strips (see figure below):
let
\begin{align*}
&D_0= [-t,t ]\times [-y_R, y_R]\\
& D_+= [t, s]\times [-y_R,y_R]\\
&D_-=[-s, -t]\times [- y_R, y_R]\\
& D_e^+ = [s,R/2] \times \mr\\
& D_e^-= [-R/2,-s]\times \mr\\
& D_1= [-s, s]\times\([y_R, y_R+R] \cup (-R- y_R, - y_R]\)
\\ & D_\infty=
 [-s,s]\times \( [R+ y_R, + \infty) \cup (-\infty,-R -y_R]\).\end{align*}

 \begin{figure}[h,c]
 \begin{center}
 \label{fig}
 \includegraphics[height=7cm]{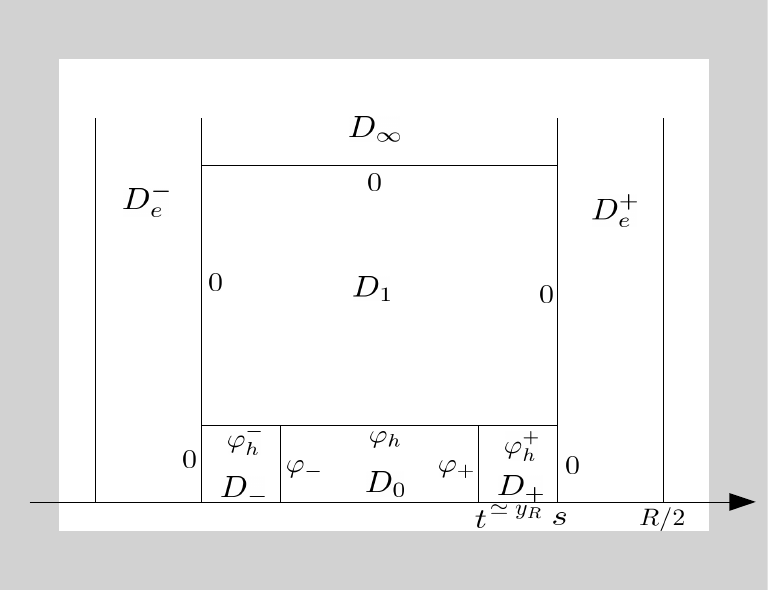}
 \end{center}
\edd{ \caption{The splitting of the strip $I_R\times\mr$ into subdomains $D_0$, $D_\pm$, $D_e^\pm$, $D_1$, $D_\infty$ and the boundary data next to each portion of the boundary between the regions.}
} \end{figure}

First we let $\j_R=\j$ in $D_0$, $\j_R=0$ in $D_\infty$  and below
we are going to define $\j_R$ on each of the other sets.

Recall that from Corollary~\ref{sym} we have that $\j(x,y)$ is the reflection of $\j(x,-y)$ with respect to the line $\{y=0\}$. We denote by  $\vp_+$  the trace $\j \cdot \vec{\nu}$  on the right-hand side of $D_0$ where $\vec{\nu}$ is the outward-pointing normal to $D_0$,  $\vp_-$ the same on the left-hand side, $\vp_h$ the trace on the upper side of $D_0$  (which by symmetry of the problem with respect to the real axis is equal
 to that on the lower side). From \eqref{cotevert}   and the  Cauchy-Schwarz inequality, we have
\begin{equation}\label{intg+}
 \int |\vp_-|^2 + \int |\vp_+|^2\le \frac C\ep  \qquad \int |\vp_-| + \int |\vp_+| \le C\sqrt\frac{y_R}\ep. \end{equation} and from Lemma \ref{harmoni}  and \eqref{yr} we have
 \begin{equation}\label{gh}
   \int_{[-t,t]\times\{y_R\}}   |\vp_h|^2\le \frac{C \ep^{10} R  (R+y_R) }  {y_R^2} < C\ep^{4}
       \qquad \int_{[-t,t]\times\{y_R\}} |\vp_h|\le C \ep^2 \sqrt{R}.
 \end{equation}
 In addition,
 integrating the relation $\div \j = 2\pi (\sum_{p\in \Lambda} \delta_p-\delta_\mr)$ over $D_0$ gives
  \begin{equation}\label{intborr}
\frac{1}{2\pi} \( \int \vp_+ +\int \vp_-  + 2 \int \vp_h\)= \nu ([-t,t]) -2t.\end{equation}\medskip

\noindent
{\it - Step 2: defining $\j_R$ in $D_+$.} We first define $\vp_0 := \j_R\cdot \vec{\nu}$ on the boundary $\p D_+$, where $\vec{\nu}$ is the outward normal to $D_+$. For a certain constant $\vp_h^+$ to be chosen later we let
\begin{equation}\label{dnu}
\vp_0 = \begin{cases} -\vp_+ & \text{on $\p D_+\cap \p D_0$,}\\
                        0 & \text{on $\p D_+\cap\{x=s\}$,}\\
                    \vp_h^+& \text{on $\p D_+\cap \{y = \pm y_R\}$.}
        \end{cases}
\end{equation}
Then, inside $D_+$, we let $\j_R = \j_1+\j_2$, where
\begin{equation}\label{curlj1}
\left\{\begin{aligned}
\div \j_1 &= 2\pi \( \sum_{i=1}^{n_+} \delta_{x_i}- m_+\delta_{\mr}\) && \text{in} \ {D_+}\\
\j_1\cdot\nu &= 0 && \text{on} \ \p D_+.
\end{aligned}
\right.,
\end{equation}
\begin{equation}\label{curlj2}
\left\{\begin{aligned}
\div \j_2 &= 2\pi (m_+-1)\delta_{\mr} && \text{in} \ {D_+}\\
\j_2\cdot\nu &= \vp_0 && \text{on} \ \p D_+.
\end{aligned}
\right.
\end{equation}
Here, $n_+$ is an integer and $m_+$ a real number which are defined by
\begin{equation}\label{n+}n_+ = \left[(s-t) - \frac1{2\pi}\(\int\vp_+ + \int\vp_h\)\right],\quad m_+ = \frac{n_+}{s-t},\end{equation}
and for $1\le i\le n_+$ we have let
$$x_i = t + \frac{s-t}{n_+}\(i+\hal\).$$
Note that the above equations do not yield a uniquely defined $\j_1$ and $\j_2$. For \eqref{curlj1} to make sense we need $n_+\ge 0$ while for \eqref{curlj2} to have a solution we need to have
\begin{equation}\label{conditions}
2\pi \(n_+ - (s-t)\)  = \int \vp_0 = 2(s-t)\vp_h^+ -  \int \vp_+,
\end{equation}
which we take as the definition of $\vp_h^+$. The fact that $n_+\ge 0$ follows for $R$ large enough depending on $\ep$ from the fact that $s-t \ge \ep^3 R$ and \eqref{intg+}, \eqref{gh}.
\medskip

\noindent
{- \it Step 3: Estimating  the energy of $\j_R$ in $D_+$, $D_-$}. To compute the renormalized energy $W(\j_R, \indic_{D_+})$ we need to define $\j_1$ and $\j_2$ more precisely. For $\j_1$
let us consider $n_+$ identical squares $\{K_i\}_{i=1}^{n_+}$ with sidelength $\frac{s-t }{n_+}  = \frac{1}{m_+} $, sides parallel to the axes and such that $K_i$ is centered at $x_i$. We define $\j_1$ restricted to $K_i$ by applying Lemma~\ref{lemrect2} with $m = m_+$ and taking $\j_1=-\nab f$,  while outside $\cup_i K_i$ we let $\j_1 = 0$. Since from Lemma~\ref{lemrect2} we have ${\j_1}_{|K_i}\cdot\vec{\nu} = 0$ on $\p K_i$, it holds that $\div\j_1 = \sum_i \div {\j_1}_{|K_i}$ and therefore \eqref{curlj1} is satisfied by $\j_1$. On the other hand, still from Lemma~\ref{lemrect2} we obtain by summing the bounds \eqref{nrjf2} and \eqref{nrjfp} on the $n_+$ rectangles
\begin{equation}\label{nrjj1}
\lim_{\eta \to 0} \left|\int_{D_+ \backslash \cup_i B(x_i, \eta) } |\j_1|^2 + 2\pi \log \eta\right|\le n_+\(C - \pi\log m_+\),\end{equation}
and
\begin{equation}\label{nrjj1b}
\forall 1<q<2,\quad \int_{D_+}| \j_1|^q \le C_q  n_+. \end{equation}

We define  $\j_2$ by applying Lemma~\ref{lem48} in $D_+$,  hence with $L = s-t$, with the boundary data $\vp_0$ and  constant weight $m = 2\pi(m_+-1)$. From \eqref{dnu} and \eqref{conditions} the hypothesis $\int  m(x) \, dx = - \int_{\p D_+} \vp$ is satisfied and applying the lemma yields
\begin{equation}\label{j20}
\forall 2\le q<4,\quad \int_{D_+} |\j_2|^q \le C_{q} \( |m_+-1|^q (s-t)^2  + (s-t)^{2-\frac{q}{2}} \|\vp_0\|_{L^2 (\p D_+)}^q\).\end{equation}
Using Lemma~\ref{lemsum} we have, recalling that $\j_R :=\j_1+\j_2$ in $D_+$ and using \eqref{nrjj1}, \eqref{nrjj1b}, \eqref{j20}:
\begin{multline}\label{prew}
W(\j_R,\indic_{D_+}) \le
C n_+ +
C_q {n_+}^{1/q} \(|m_+-1| (s-t)^{2/q'} +(s-t)^{2/q'-1/2}\|\vp_0\|_{L^2 (\p D_+)}\) \\ +
C \( |m_+-1|^2 (s-t)^2  + (s-t) \|\vp_0\|_{L^2 (\p D_+)}^2\),
\end{multline}
for any $1<q<2$ such that the conjugate exponent $q'$ is less then $4$. Now, from \eqref{n+}, using \eqref{intg+}, \eqref{gh}, \eqref{yr} and the fact that $y_R\le s-t\le y_R+1$ we deduce  that
\begin{equation}\label{pren}|n_+- (s-t)|\le C\(\ep^2\sqrt R+ \sqrt\frac{y_R}\ep\), \quad |m_+ - 1|\le C\(\frac1{\ep\sqrt R} + \frac1{\sqrt{\ep y_R}}\),\end{equation}
and thus for $R$ large enough depending on $\ep$, since $s-t\simeq y_R$ for large $R$ and using \eqref{yr} again as well as $\ep<1$,
\begin{equation}\label{npcr}n_+\le C\ep^{5/2}R, \quad |m_+ - 1|\le \frac C{\ep^2\sqrt R}.\end{equation}
Moreover, from \eqref{n+}, \eqref{conditions} and \eqref{gh} we find
\begin{equation}\label{vphp}|\vp_h^+| \le \frac{1}{2(s-t)}\left|2\pi +\int\vp_h\right| \le  C\frac{\ep^2\sqrt R}{y_R}\le \frac C{\ep\sqrt R}.\end{equation}
Then, in view of \eqref{dnu}, \eqref{intg+},
$$\|\vp_0\|_{L^2 (\p D_+)}^2\le \frac C\ep.$$
Now we fix for instance $q = 3/2$, so that $q' = 3$ and combining the above with \eqref{prew}, \eqref{npcr} we find that for $R$ large enough depending on $\ep$, and denoting by $C_\ep$ a positive constant depending on $\ep$ but independent of $R$,
$$W(\j_R,\indic_{D_+}) \le C \ep^{5/2}R + C_{\ep} R^{2/3} \times R^{2/3-1/2} + C\ep^{3/2} R.$$
Thus for $R$ large enough depending on $\ep$ we find that
\begin{equation}\label{nrjjr}
W(\j_R, \indic_{D_+}) \le C\ep^{3/2}  R.\end{equation}
An almost symmetric computation yields the same bound for $W(\j_R, \indic_{D_-})$. It suffices to let
\begin{equation}\label{n-}n_- = 2(s-t) - \frac1{2\pi}\(\int\vp_++\int\vp_-+2\int\vp_h\) - n_+,\end{equation}
and carry on the proof with minuses instead of pluses. The fact that $n_-$ is an integer follows from the identity \eqref{intborr}, the fact that $2s$ is an integer and the fact that $\nu([-t,t])\in\mn$. Moreover the definition of $n_-$ implies that
$$  n_- = \left[(s-t) - \frac1{2\pi}\(\int\vp_- + \int\vp_h\)\right],\quad \text{or}\quad n_- = \left[(s-t) - \frac1{2\pi}\(\int\vp_- + \int\vp_h\)\right]+1,$$ hence $n_-$ is positive if $R$ is large enough and  \eqref{pren} holds for $n_-$ as well. The rest of the proof is unchanged. \medskip

\noindent
{- \it Step 4: defining  $\j_R$ over $D_1$.}  We need only consider the intersection of $D_1$ with the upper half-plane (and then extend by reflection).
We let $\vp_0$ be equal to $-\vp_h^-$, $-\vp_h$, $-\vp_h^+$, respectively, on  the intersection of  $\p D_1$ with $\p D_-$, $\p D_0$, $\p D_+$, respectively. On the remaining three sides of $\p D_1$ we let $\vp_0 = 0$. From \eqref{conditions} and its equivalent for $n_-$ and the fact that $n_\pm = (s-t)m_\pm$ we have
$$- \int \vp_0 = \pi(n_++n_- - 2(s-t)) +\hal\int\vp_++\hal\int\vp_-+\int\vp_h,$$
and then \eqref{n+}, \eqref{n-} imply that the integral of $\vp_0$ is zero.

Thus there exists a harmonic function $u$ in $D_1$ with normal derivative $\vp_0$ on $\p D_1$, we let $\j_R = \nab u$ on $D_1$. Using  \eqref{vphp} --- which holds for $\vp_h^-$ as well ---  and \eqref{gh} we have 
\edd{$$\int_{[-t,t]\times\{y_R\}} |\vp_h|^2 + \int_{[t,s]\times\{y_R\}} |\vp_h^+|^2 \le C\ep^4+ C \frac{\ep^4 R}{y_R},$$ }
hence
\begin{equation}\label{normeg}
\|\vp_0\|_{L^2(\p D_1) }^2 \le C\(\ep^4 + \ep^4 R/y_R\)\le C \ep.
\end{equation}
Then standard elliptic estimates yield as in Lemma~\ref{lem48} that
\begin{equation}\label{estjkh} \int_{D_1} |\j_R|^2 \le C R \|\vp_0\|_{L^2}^2\le CR\ep,\end{equation}
where we have concluded by \eqref{yr}.\medskip

\noindent
{\it  - Step 5: defining $\j_R$ over $D_e^+$}. The construction will be entirely parallel in $D_e^-$. We note that $D_e^+$ is an infinite strip of width $R/2- s$  and we have chosen $s$ so that this quantity is an integer. We can thus split this strip into exactly $R/2-s$ strips of width $1$. On each of these strips we define $\j_R$ to be equal to $0$ for $|y|\ge \frac{1}{2}$ and for $|y|\le \hal $ (i.e. in a square of sidelength $1$) we choose it to be $\nab f$ where $f$ is given by Lemma \ref{lemrect2} applied with $m=1$. Since $\j_R \cdot\vec{ \nu} = 0$ on the boundary of each of these squares, no divergence  is created at the interfaces, and the resulting $\j_R$ satisfies $\div \j_R = 2\pi (\sum_{p \in \Lambda }\delta_p - \delta_\mr)$. In addition in view of \eqref{nrjfp} the cost in energy is equal to a constant  times the number of strips, that is
\begin{equation}\label{nrjde}
W(\j_R, \indic_{D_e}) \le C |R/2-s|\le C \ep R.\end{equation}\medskip

\noindent
{\it  - Conclusion.} We have now defined $\j_R$ over the whole strip ${I}_R\times \mr $. It satisfies items ii) and iii). The main point is again that as long as $\j \cdot \vec{ \nu}$ is continuous across an interface it creates no singular divergence there.
Combining \eqref{nrjde} with \eqref{438}, \eqref{nrjjr} and \eqref{estjkh}, $\j_R$  also satisfies  \eqref{restronque}.
 This concludes the proof of Proposition \ref{tronque}.

\noindent
{\sc Etienne Sandier}\\
Universit\'e Paris-Est,\\
LAMA -- CNRS UMR 8050,\\
61, Avenue du G\'en\'eral de Gaulle, 94010 Cr\'eteil. France\\
\& Institut Universitaire de France\\
{\tt sandier@u-pec.fr}\\ \\
{\sc Sylvia Serfaty}\\
UPMC Univ  Paris 06, UMR 7598 Laboratoire Jacques-Louis Lions,\\
 Paris, F-75005 France ;\\
 CNRS, UMR 7598 LJLL, Paris, F-75005 France \\
 \&  Courant Institute, New York University\\
251 Mercer st, NY NY 10012, USA\\
{\tt serfaty@ann.jussieu.fr}

\end{document}